\documentclass{amsart}
\usepackage{fullpage}
\usepackage{amsfonts, amsmath, amssymb, amscd, amsthm}
\usepackage{color}
\usepackage{rotating}
\usepackage{graphics}
\RequirePackage[bookmarks,%
colorlinks,%
urlcolor=prune,%
citecolor=red,%
linkcolor=blue,%
hyperfigures,%
pdfcreator=LaTeX,%
breaklinks=true,%
pdfpagelayout=SinglePage,%
bookmarksopen=true,%
bookmarksopenlevel=2]{hyperref}
\usepackage{float}
\usepackage{caption} 
\usepackage{subcaption}
\makeatletter
\renewcommand{\tocsection}[3]{%
	\indentlabel{\@ifnotempty{#2}{\bfseries\ignorespaces#1#2.\quad}}\bfseries#3}

\newcommand{\R}{\mathbb R}

\newcommand{\N}{\mathbb N}

\newcommand{\dsp}{\displaystyle}
\newcommand{\eps}{\varepsilon}

\newtheorem{proposition}{Proposition}
\newtheorem{theorem}{Theorem}
\newtheorem{definition}{Definition}
\newtheorem{remark}{Remark}

\newtheorem{lemma}{Lemma}

\numberwithin{equation}{section}

\definecolor{darkspringgreen}{rgb}{0.09, 0.45, 0.27}
\definecolor{prune}{rgb}{0.44, 0.11, 0.11}

\begin{document}
\title{Mathematical modeling and numerical analysis for the higher order Boussinesq system}
\date{\today}
%
%
\author{Bashar Khorbatly}
\address{Lebanese American University (LAU), Graduate Studies and Research (GSR) office, School of Arts and Sciences, Computer Science and Mathematics Department, Byblos, Lebanon}
\email{bashar-elkhorbatly@hotmail.com}
\author{Ralph Lteif}
\address{Lebanese American University (LAU), Graduate Studies and Research (GSR) office, School of Arts and Sciences, Computer Science and Mathematics Department, Beirut, Lebanon}
\email{Corresponding author, ralph.lteif@lau.edu.lb}
\author{Samer Israwi}
\address{Lebanese University, Laboratory of Mathematics-EDST, Department of Mathematics, Faculty of Sciences 1, Beirut, Lebanon}
\email{s$\_$israwi83@hotmail.com}
\author{St\'ephane Gerbi}
\address{Laboratoire de Math\'ematiques UMR 5127 CNRS \& Universit\'e de Savoie Mont Blanc, Campus scientifique, 73376 Le Bourget du Lac Cedex, France}
\email{stephane.gerbi@univ-smb.fr}
\subjclass[2010]{35Q35, 35L45, 35L60, 76B45, 76B55, 35C07, 65L99} 
\keywords{Water waves, Boussinesq system, higher-order asymptotic model, well-posedness, traveling waves, explicit solution,
numerical validation.}

\date{\today}

\begin{abstract}
This study deals with higher-ordered asymptotic equations for the water-waves problem. We considered the higher-order/extended Boussinesq equations 
over a flat bottom topography in the well-known long wave regime. Providing an existence and uniqueness of solution on a relevant time scale of order $1/\sqrt{\eps}$ 
and showing that the solution's behavior is close to the solution of the  water waves equations with a better precision corresponding to initial data, 
the asymptotic model is well-posed in the sense of Hadamard. Then we compared several water waves solitary solutions with respect to the numerical solution of our model. 
At last, we solve explicitly this model and validate the results numerically.
\end{abstract}
\maketitle
\tableofcontents
\section{Introduction}
\subsection{The water-wave equations.}
In this paper, we investigate the one-dimensional flow of the free surface of a homogeneous, immiscible fluid moving above a flat topography $z=-h_0$. 
	The horizontal and vertical variables are denoted respectively by $x \in \R$ and $z \in \R$ and $t \geq 0$ stands for the time variable. 
	The free surface is parametrized by the graph of the function $\zeta(t,x)$ denoting the variation with respect to its rest state $z=0$ (see Figure~\ref{flattopdom}).
	The fluid occupies the strictly connected ($\zeta(t,x) + h_0 >0$) domain $\Omega_t$ at time $t\geq 0$ denoted by:
	$$ \Omega_t= \{ (x,z) \in \R^2; \ -h_0 \leq z \leq \zeta(t,x) \}.$$ 
	\begin{figure}[H]
		\centering
		\includegraphics[width=1.0\textwidth]{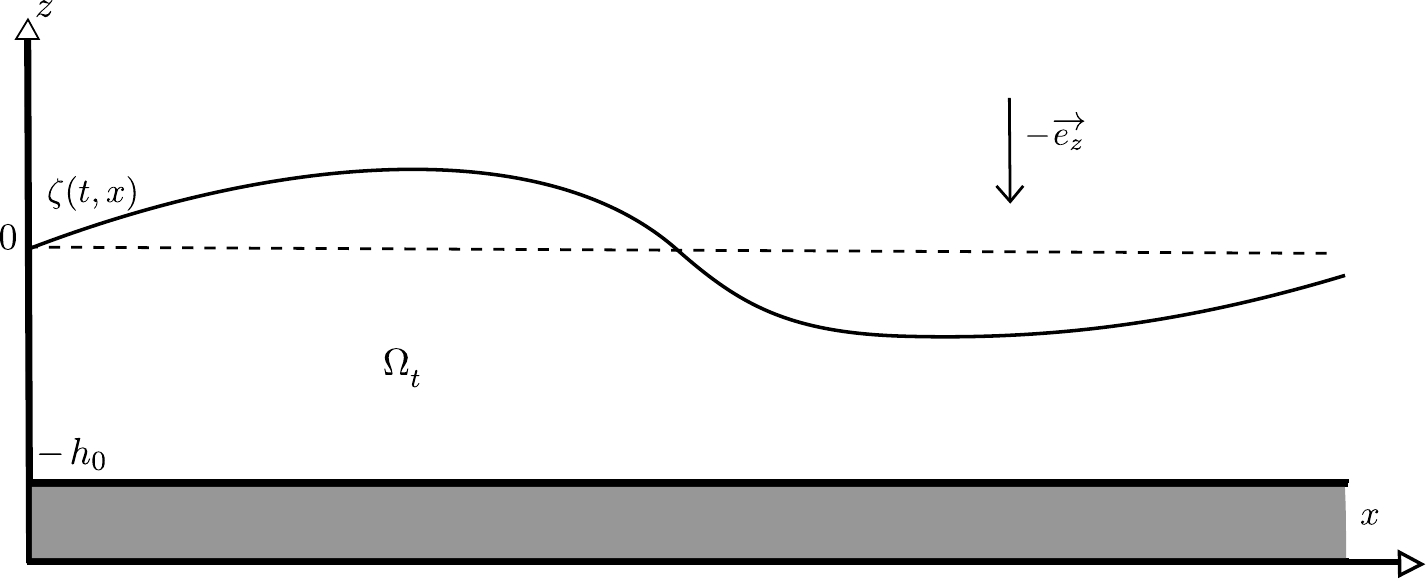}
		\caption{One-dimensional flat bottom fluid domain.}
		\label{flattopdom}
	\end{figure}
	\noindent The fluid is considered to be perfect, that is with no viscosity and only affected by the force of gravity.
	We also assume the fluid to be incompressible and the flow to be irrotational so that the velocity field is divergence and curl free.
	We denote by $(\rho,V)$ the constant density and velocity field of the fluid. The first boundary condition at the free surface expresses a balance of forces.
	Kinematic boundary conditions are considered assuming that both the surface and bottom are impenetrable, that is no particle of fluid can cross.
	The set of equations describing the flow is now complete and is commonly known as the \emph{full Euler} equations:
	\begin{equation}
		\left\{
		\begin{array}{lcl}
			\displaystyle\partial_t V+V\cdot\nabla_{x,z} V = -g\overrightarrow{e}_z-\dfrac{\nabla_{x,z} P}{\rho} & \hbox{in} & (x,z)\in \Omega_t, \ t\geq 0 \vspace{1mm},\\
			\displaystyle\nabla_{x,z}\cdot V=0 & \hbox{in} & (x,z)\in \Omega_t, \ t\geq 0  \vspace{1mm},\\
			\displaystyle\nabla_{x,z}\times V=0 & \hbox{in} & (x,z)\in \Omega_t, \ t\geq 0    \vspace{1mm},\\
			P|_{z=\zeta(t,x)}=0 & \hbox{for} &  t\geq 0,\ x \in \R,\\ 
			\displaystyle\partial_t \zeta-\sqrt{1+\vert \partial_x \zeta \vert^2} n_{\zeta}\cdot V|_{z=\zeta(t,x)} =0 & \hbox{for} & t \geq 0\vspace{1mm}, \ x \in \R,\\
			\displaystyle -V\cdot \overrightarrow{e_z} = 0 & \hbox{at} & z=-h_0 \vspace{1mm}, \ t\geq 0,
			\vspace{1mm},\\
\displaystyle \lim_{\vert(x,z)\vert\to\infty}\vert\zeta(x,z)\vert+\vert V(t,x,z)\vert=0 & \hbox{in} & (x,z)\in \Omega_t, \ t\geq 0 \; .
		\end{array}
		\right.
		\label{euler}
	\end{equation}
	where $n_{\zeta}=\dfrac{1}{\sqrt{1+|\partial_x \zeta|^2}} (-\partial_x \zeta, 1)^T$ denotes the upward normal vector to the free surface.
	
	The theoretical study of the above system of equations is extremely difficult due to its large number of unknowns and its time-dependent moving domain $\Omega_t$.
	In fact, we have a free boundary problem, in other words the domain is itself one of the unknowns. 
Using the assumption of irrotational velocity field, one can express the latter as the gradient of a potential function $\varphi$. 
	This potential satisfies the Laplace equation inside the fluid, $\Delta_{x,z} \varphi =0$ in $(x,z) \in \Omega_t$.
	Consequently, the evolution of the velocity potential is written now using Bernoulli's equation.
	Although the system now is simpler, a free boundary problem still exists. To get over this obstacle, Craig and Sulem~\cite{CS93,CSS92} had an interesting idea following Zakharov work~\cite{Zakharov68}, consisting of a reformulation of the system of equations~\eqref{euler} using the introduction of a Dirichlet-Neumann operator,
	thus reducing the dimension of the considered space and the unknowns number. Denoting by $\psi$ the trace of the velocity potential at the free surface, $\psi(t,x) =\varphi (t,x,\zeta(t,x))=\varphi _{| z=\zeta}$, the Dirichlet-Neumann operator is introduced
	\begin{equation*}\label{diriclet}
		\mathcal{G}[\zeta]\psi = -\big(\partial_x\zeta\big)\cdot\big(\partial_x\varphi\big)_{\mid_{z=\zeta}} + \big(\partial_z \varphi\big)_{\mid_{z=\zeta}} = \sqrt{1 + \big\vert\partial_x\zeta\big\vert ^2}\big(\partial_n\varphi\big)_{\mid_{z=\zeta}}
	\end{equation*}
	where $\varphi$ is defined uniquely from $(\zeta,\psi)$ as a solution of the following Laplace problem (see~\cite{Lannes2013} for a complete and accurate analysis):
	\begin{equation*}
		\left\{
		\begin{array}{lcl}\label{BVP1}
			\displaystyle\partial_x^2 \varphi+ \partial_z^2 \varphi = 0 & \hbox{in} & - h_0  < z < \zeta(t,x),\\
			\displaystyle\partial_z\varphi_{\mid_{z=-h_0}}=0,\\
			\displaystyle\varphi_{\mid_{z=\zeta}}=\psi(t,x) .
		\end{array}
		\right.
	\end{equation*}
	with $\partial_n = n. \nabla_{x,z}$ 
	the normal derivative in the direction of the concerned vector $n$. Thus, the evolution of only the two variables $(\zeta,\psi)$ located at the free surface characterize the flow. This system is known by the Zakharov/Craig-Sulem formulation of the water-waves equations giving :
\begin{equation}\label{Za}
\left\{
\begin{array}{lcl}
\displaystyle\partial_t \zeta-\frac{1}{\mu}\mathcal{G} [ \zeta]\psi= 0 \; ,\\
\displaystyle\partial_t\psi+\zeta+\frac{1}{2}\vert\partial_x \psi\vert^2 - \displaystyle\frac{( \mathcal{G}_{\mu}[ \zeta]\psi+\partial_x(\zeta)\cdot\partial_x\psi)^2}{2(1+ \vert\partial_x\zeta\vert^2)}= 0 \; .
\end{array}
\right.
\end{equation}
The above system of equations has a particularly rich structure, and depending on the physical properties of the flow, it is possible to obtain solutions to \eqref{Za} with different qualitative properties. Nonlinear effects, for example, become more important as wave amplitude increases. Although Zakharov's reformulation resulted in a reduced system of equations, the description of these solutions from a qualitative and quantitative point of view remains very complex.
	A remedy for this situation requires the construction of simplified asymptotic models whose solutions are approximate solutions of the full system. 
	These approximate models allow to describe in a fairly precise way the behavior of the complete system in a specific physical regime.
	This requires a rescaling of the system in order to reveal small dimensionless parameters which allow to perform asymptotic expansions of non-local operators (Dirichlet-Neumann),
	thus ignoring the terms whose influence is minimal.
	The order of magnitude of these parameters makes it possible to identify the considered physical regime. 
	We start by introducing respectively the commonly known nonlinear and shallowness parameters:
\begin{equation*}
\varepsilon=\frac{a}{h_0}=\frac{\text{amplitude of the wave}}{\text{reference depth}} \; , \qquad\qquad\sqrt{\mu}=\frac{h_0}{\lambda}=\frac{\text{reference depth}}{\text{wave-length of the wave}} \; ,
\end{equation*}
where $0\leq\varepsilon\leq 1$ is often called nonlinearity parameter, while $0\leq\mu \leq 1$ is called the shallowness parameter. 
In this manner, the dimensionless formulation of \eqref{Za} reads:
\begin{equation}\label{Zakharovv}
\left\{
\begin{array}{lcl}
\displaystyle\partial_t \zeta-\frac{1}{\mu}\mathcal{G}_{\mu}[\varepsilon\zeta]\psi= 0 \; ,\\
\displaystyle\partial_t\psi+\zeta+\frac{\varepsilon}{2}\vert\partial_x \psi\vert^2 -\varepsilon\mu\displaystyle\frac{(\frac{1}{\mu}\mathcal{G}_{\mu}[\varepsilon\zeta]\psi+\partial_x(\varepsilon\zeta)\cdot\partial_x\psi)^2}{2(1+\varepsilon^2\mu\vert\partial_x\zeta\vert^2)}= 0 \; ,
\end{array}
\right.
\end{equation}
where $\psi (t,x) =\varphi_{\mid_{z=\eps \zeta}}$ and $\mathcal{G}_\mu[\varepsilon\zeta]\psi = \sqrt{1 + \mu\eps^2\big\vert\partial_x\zeta\big\vert ^2}\big(\partial_n\varphi\big)_{\mid_{z=\varepsilon\zeta}}$.

Let us now identify the asymptotic geophysical shallow-water ($\mu\ll1$) category (or sub-regime) associated with our work. An additional assumption is made on the nonlinearity parameter, from which a diverse set of asymptotic models can be derived. More precisely, it is possible to deduce from \eqref{Zakharovv} a (much simpler) asymptotic model that is more amenable to numerical simulations and have more transparent properties. For instance, taking $\eps\sim\mu$ into account, the flow under consideration is said to be in a small amplitude regime.

\subsection{Shallow-water, flat bottom, small amplitude variations $(\mu\ll1, \eps\sim\mu)$.}

In this paper, we restrict our work on the well-known long waves regime with a flat topography for which the "original" or "standard" Boussinesq system can be derived. Defining the depth-averaged horizontal velocity by : 
\begin{equation}\label{defvelocity}
v(t,x)=\frac{1}{1 + \varepsilon\zeta (t,x)}\int_{-1}^{\varepsilon\zeta (t,x)}\partial_x\varphi(t,x,z)\hspace{0.1cm}dz \; ,
\end{equation}
under the extra assumption $\eps\sim\mu$, we can neglect the terms which are of order $\mathcal{O}(\mu^2)$ in the Green-Naghdi equations (we refer to \cite{GN76,GLN74} for formal derivation and to \cite{Israwi2011,Israwi2010,Khorbatly2021,AMBP_2018__25_1_21_0} for well-posedness); then the standard Boussinesq equations reads:
\begin{equation}\label{standard-bunsq}
\left\{
\begin{array}{lcl}
\displaystyle\partial_t\zeta+\partial_x\big( (1+\eps\zeta )v \big)=0\vspace{1mm}\; ,\\
\displaystyle (  1 - \eps \frac{1}{3} \partial_x^2 )  \partial_t v + \partial_x\zeta + \varepsilon v\partial_x v  =\mathcal{O} (\eps^2) \; .
\end{array}
\right.
\end{equation}
Many strategies exist to study the water-wave problem especially by deriving equivalent models with better mathematical structure such as well-posedness, 
conservation of energy, solitary waves, or physical properties (see for instance \cite{BBM72,LPS2012,BCL2005,Chazel2007,MSZ2012,SX2012,SWX2017,Burtea2016-1,Burtea2016-2,Lannes2013,Saut-Li,SAUT20202627,saut2021}).
 It is worth noticing  that the well posed results for such model exist on a time scale of order $1/\sqrt{\eps}$ (methods based on dispersive estimate in \cite{Zakharov68}) 
 and $1/\eps$ (energy estimate method  in \cite{Lannes2013} ). A better precision is obtained when the $\mathcal{O}(\mu^2)$ 
 terms are kept in the equations: only $\mathcal{O}(\mu^3)$ terms are dropped. Following the work in a  series of papers on the extended Green-Naghdi equations
 \cite{Matsuno2015,Matsuno2016,KZI2018,KZI2021}, one may write the extended Boussinesq  equations by incorporating higher order dispersive effects as follows:
\begin{equation}\label{ex-boussinesq}
\left\{
\begin{array}{lcl}
\displaystyle\partial_t\zeta+\partial_x(hv)=0\vspace{1mm}\; ,\\
\displaystyle (  1+\eps\mathcal{T}[\zeta]+\eps^2\mathfrak{T} )\partial_t v + \partial_x\zeta+\eps  v \partial_x v   +\eps^2 \mathcal{Q}v =\mathcal{O} (\eps^3) \; ,
\end{array}
\right.
\end{equation}
where $h=1+\eps\zeta$ is the non-dimensionalised height of the fluid and we denote the three operators :
\begin{equation*}
\mathcal{T}[\zeta]w =-\frac{1}{3h}\partial_x\big((1+3\eps\zeta)\partial_xw\big),
\quad\mathfrak{T} w = -\frac{1}{45}\partial_x^4w ,  \quad\mathcal{Q}v = -\frac{1}{3}\partial_x\big(vv_{xx}-v_x^2\big) \; .
\end{equation*}


\subsection{Presentation of the results}
As mentioned before, we will first derive an extended Boussinesq  equations in the same way as the derivation of the extended Green-Naghdi equations: we will keep every terms up to the third order in $\eps$. This is done in the next section,  section \ref{model}. 
Section \ref{justification} is devoted to the full justification of the extended Boussinesq system. We will firstly,
in subsection \ref{quasilinear}, write the extended Boussinesq system in a quasilinear form. The linear analysis, performed in subsection \ref{linear-analysis} will permit by the
energy estimate method to state, in the subsection \ref{mainresults}, the main results of well-posedness, stability and convergence of the proposed extended Boussinesq system.

As for usual Green-Naghdi and Boussinesq model, we are interested in the construction of a solution as a solitary wave. 
We will prove in section \ref{solitary-approx} that the profile of this solitary wave is a solution of a 3rd order non linear ordinary differential equation, ODE. Thus, it seems
impossible to find an explicit form of this profile. Therefore, we will compute, using Matlab ODE solver \texttt{ode45}, an approximate profile.
We will compare the obtained solutions with the solutions of water-waves equations and find that this solution is a better approximation than the solution of the original
Green-Naghdi equation. 

Lastly, instead of finding an analytical exact solitary wave, we will find an explicit solution with correctors in section \ref{explicit-solitary}.


\subsection{Comments on the results.}
In this section we try to highlight the potential need of higher-ordered models and their benefits over the classical asymptotic ones.
Despite having a more complicated structure than classical models,
higher ordered models may still be considered simpler than the original full Euler system~\eqref{euler}. In fact, as opposed to the full Euler system, these high order models enjoy a reduced structure in terms of number of equations, unknown numbers and dimension space which make them more suitable for theoretical and numerical study. 
Moreover, higher order approximations may have similar well-posedness results as classcial ones on relevant time scales due to standard mathematical tools. Based on section \ref{justification} and previous works \cite{KZI2018,KZI2021} this can be concluded at least in the one-dimensional case. 
However, the advantage is obvious in terms of controlling the convergence precision of the approximation error with respect to Euler equations (see in particular Theorem \ref{convergence} of section \ref{justification}).\\
On the other hand, while the solitary wave profile cannot be derived explicitly for higher order approximations, the numerical solution fits the corresponding one of the original Euler system much better than classical models (as shown by figure \ref{SWcomp}). The numerical solution computation requires simple discretization of a third-order nonlinear ODE using Matlab \texttt{ode45} solver. Furthermore, it is noteworthy that by removing the $\varepsilon^2$ extended-Boussiseq ODE terms, the Green-Naghdi's ODE can be recovered.

\subsection{Notation.}
We denote by $C(\lambda_1, \lambda_2,...)$ a constant depending on the parameters 
$\lambda_1$, $\lambda_2$, ... and \emph{whose dependence on the $\lambda_j$ is always assumed to be nondecreasing}. The notation $a\lesssim b$ means that $a\leq Cb$, 
for some non-negative constant $C$ whose exact expression is of no importance (\emph{in particular, it is independent of the small parameters involved}).

We denote the $L^2$ norm $\vert\cdot\vert_{L^2}$ simply by $\vert\cdot\vert_2$. The inner product of any functions $f_1$
and $f_2$ in the Hilbert space $L^2(\R^d)$ is denoted by
$
(f_1,f_2)=\int_{\R^d}f_1(X)f_2(X) dX.
$ The space $L^\infty=L^\infty(\R^d)$ consists of all essentially bounded, Lebesgue-measurable functions
$f$ with the norm
$
\vert f\vert_{L^\infty}= \hbox{ess}\sup \vert f(X)\vert<\infty
$. 
We denote by $W^{1,\infty}(\R)=\big\lbrace f\in L^\infty,  f_x\in L^{\infty}\big\rbrace$ endowed with its canonical norm.

For any real constant $s$, $H^s=H^s(\R^d)$ denotes the Sobolev space of all tempered
distributions $f$ with the norm $\vert f\vert_{H^s}=\vert \Lambda^s f\vert_2 < \infty$, where $\Lambda^s$ 
is the pseudo-differential operator $\Lambda^s=(1-\partial_x^2)^{s/2}$.

For any functions $u=u(t,X)$ and $v(t,X)$
defined on $[0,T)\times\R^d$ with $T>0$, we denote the inner product, the $L^p$-norm and especially
the $L^2$-norm, as well as the Sobolev norm, with respect to the spatial variable, by $(u,v)=(u(\cdot,t),v(\cdot,t))$, $\vert u \vert_{L^p}=\vert u(\cdot,t)\vert_{L^p}$, $\vert u \vert_{L^2}=\vert u(\cdot,t)\vert_{L^2}$, 
and $ \vert u \vert_{H^s}=\vert u(\cdot,t)\vert_{H^s}$, respectively.

Let $C^k(\R^d)$ denote the space of $k$-times continuously differentiable functions.
For any closed operator $T$ defined on a Banach space $Y$ of functions, the commutator $[T,f]$ is defined by $[T,f]g=T(fg)-fT(g)$ with $f$, $g$ and $fg$ belonging to the domain of $T$.
\section{The higher-order/extended Boussinesq equations}\label{model}
When the surface elevation is of small amplitude, that is, when an assumption is made on the nonlinearity parameter, the extended Green-Naghdi equations \cite{Matsuno2015, Matsuno2016, KZI2018, KZI2021} can be greatly simplified. Based on this, the extended Boussinesq with $\eps\sim\mu$ reads for one-dimensional small amplitude surfaces:
\begin{equation}\label{original-bous}
\left\{
\begin{array}{lcl}
\displaystyle\partial_t\zeta+\partial_x(hv)=0\vspace{1mm}\; ,\\
\displaystyle (  h +\eps\mathcal{T}[h]+\eps^2\mathfrak{T} )\partial_t v + h\partial_x\zeta+\eps h v \partial_x v   +\eps^2 \mathcal{Q}v =\mathcal{O} (\eps^3) \; ,
\end{array}
\right.
\end{equation}
where the right-hand side is of order $\eps^3$, and we see the dependence on $\eps^2$ in the left-hand side. Here $h=1+\eps\zeta$ and we denote by 
\begin{equation*}
\mathcal{T}[h]w =-\frac{1}{3}\partial_x\big( h^3 \partial_xw\big) \; , \qquad\mathfrak{T} w = -\frac{1}{45}\partial_x^4w \; ,  \qquad\mathcal{Q}v = -\frac{1}{3}\partial_x\big(vv_{xx}-v_x^2\big) \; .
\end{equation*}
\begin{remark}
Some of the components in the second equation's left-most term are of the size $\mathcal{O}(\eps^3)$. They were kept to preserve the operator's $\Im= h+\eps\mathcal{T}[h]-\eps^2\partial_x^4$ good properties; otherwise, these properties would have been disrupted  (see section \ref{invert-op}).
\end{remark}
\subsection{The modified system.} 
First of all, let us factorize all higher order derivatives (third and fifth) in the left-most term of the above system \eqref{original-bous}. In fact, we only have to factorize third-order derivatives 
and this is possible by setting $\pm\eps^2\mathcal{T}[h](vv_x)$ in the second equation. An inconvenient feature appears in this left-most term due to the positive sign in front of the 
elliptic forth-order linear operator $\mathfrak{T}$ which ravel the way towards well-posedness using energy estimate method. This obviously affect the invertibility of the factorized 
operator as we will see in section \ref{invert-op}. For this reason we proceed as in \cite{KZI2018,KZI2021} by using a $BBM$ trick represented in the following approximate 
equation $\partial_tv +\eps vv_x= -\zeta_x+O(\eps)$ to overcome this difficulty.

At this stage, it is noteworthy that from \cite{KZI2018,KZI2021} one may conclude directly the well-posedness results for such system but when the effect of surface tension is taken into consideration, the existence time scale is up to order $1/\eps$. 
This presence of the surface tension was essential for controlling higher order derivatives yielding from the BBM trick (see remarks in \cite{KZI2021}). 
In our case, the surface tension is neglected and thus we have to do proceed differently. 
The idea is to replace the capillary terms by a vanishing term $\pm\eps^2\zeta_{xxx}$ which will play a similar role. 
The term with a negative sign is used for a convenient definition of the energy space (see Definition \ref{defispace}) in such a way that the other term can be controlled. 
As a consequence, the existence time will get smaller with respect to the case of surface tension presence, \textit{i.e.} the time scale reached is up to order $1/\sqrt{\eps}$. In view of the above notes (we refer to remarks \ref{rem1} and \ref{rem2} for more details), the modified system reads:
\begin{equation}\label{boussinesq}
\left\{
\begin{array}{lcl}
\displaystyle\partial_t\zeta+\partial_x(hv)=0\vspace{1mm}\; ,\\
\displaystyle (  h+\eps\mathcal{T}[h]-\eps^2\mathfrak{T} ) \big(\partial_t v+\varepsilon vv_x\big) + h\partial_x\zeta- \eps^{2} \zeta_{xxx}+\frac{2}{45}\eps^2 \zeta_{xxxxx}    + \eps^{2} \zeta_{xxx}+\eps^2 \mathcal{Q}[U]v_x  =\mathcal{O} (\eps^3) \; ,
\end{array}
\right.
\end{equation}
where $U=(\zeta,v)$, $h(t,x)=1+\eps\zeta(t,x)$ and denote by
\begin{equation}\label{exp1}
\mathcal{T}[ h ]w =-\frac{1}{3}\partial_x( h^3 \partial_xw), \qquad\qquad \mathfrak{T} w = -\frac{1}{45}\partial_x^4w ,  \qquad\qquad  \mathcal{Q}[U]f = \frac{2}{3}\partial_x\big(v_xf\big) \; .
\end{equation}
\begin{remark}
An equivalent formulation of system~\eqref{boussinesq} has been numerically studied recently in~\cite{LG2021}. This formulation is obtained by dividing the second equation of system~\eqref{boussinesq} by the water height function, $h$ and removing time dependency from the left-most factorized operator while keeping the same precision of the model. During the numerical computations this operator  has to be inverted at each time step so one can solve system~\eqref{boussinesq}. The time dependency has to be amended in order to reduce the computational time.
\end{remark}
We state here that the solution of~\eqref{Zakharovv} is also a solution to the extended Boussinesq system \eqref{boussinesq} up to terms of order $\mathcal{O}(\eps^3)$.
\begin{proposition}[Consistency]\label{consistency}
Suppose that the full Euler system \eqref{Zakharovv} has a family of solutions $U^{euler}=(\zeta,\psi)^T$ such that there exists $T > 0$, $s>3/2$ for which $(\zeta,\psi' )^T$ is bounded in $L^{\infty}([0; T);H^{s+N})^2$ with N sufficiently large, uniformly with respect to $\eps\in(0,1)$. Define $v$ as in \eqref{defvelocity}. Then $(\zeta,v)^T$ satisfy \eqref{boussinesq} up to a remainder $R$, bounded by 
\begin{equation}\label{R}
\Vert R\Vert_{(L^{\infty}[0,T[;H^s)}\le \eps^3 C \; ,
\end{equation}
where $C=C(h_{min}^{-1}, \Vert \zeta \Vert_{L^{\infty}([0,T[;H^{s+N})}, \Vert \psi' \Vert_{L^{\infty}([0,T[;H^{s+N})})$ .
\end{proposition}
\begin{proof}
Equation one of \eqref{boussinesq} exactly coincides with that of \eqref{Zakharovv}. It remains to check that the second equation is satisfied up to a remainder $R$ such that \eqref{R} holds. For this sake,
we need an asymptotic expansion of $\psi'$ in terms of $v$ which can be deduced from the work done in \cite{KZI2018} as follows :
\begin{equation}\label{psi'}
\psi'= v -\frac{1}{3}\eps\partial_x\big((1+3\eps\zeta)v_x\big) + \eps^2\frac{1}{3}\zeta\partial_x^2v + \eps^2\mathfrak{T} v + \eps^3R_{3}^{\eps} \; .
\end{equation}
Now we proceed iusing the same arguments as the ones used in Lemmas 5.4 and 5.11 in \cite{Lannes2013} to give some control on $R_3^{\eps}$ as follows :
\begin{equation}\label{control-of-R3}
\vert R_3^{\eps}\vert_{H^s}\le C(h_{min}^{-1}, \vert\zeta\vert_{H^{s+6}}) \vert \psi'\vert_{H^{s+6}} \qquad\text{ and }\qquad \vert \partial_t R_3^{\eps}\vert_{H^s}\le C(h_{min}^{-1}, \vert\zeta\vert_{H^{s+8}}, \vert \psi'\vert_{H^{s+8}} ) \; .
\end{equation}
Then we take the derivative of the second equation of \eqref{Zakharovv} and substitute $\mathcal{G}[\varepsilon\zeta]\psi$ and $\psi'$ by $ - \eps \partial_x (hv) $ and \eqref{psi'} respectively. 
Therefore, taking advantage of the estimates  \eqref{control-of-R3} provides the control of all terms of order $\eps^3$ as in \eqref{R} with $N$ large enough (mainly greater than $8$).
\end{proof}
\section{Full justification of the extended Boussinesq system $(\mu^3<\mu^2<\mu\ll1, \eps\sim\mu)$}\label{justification}
The two main issues regarding the validity of an asymptotic model are the following:
\begin{itemize}
\item Are the Cauchy problems for both the full Euler system and the asymptotic model well-posed for a given class of initial data, and over the relevant time scale ?
\item Can the water waves solutions be compared to the solutions of the full Euler system when corresponding initial data are close? If yes, can we estimate how close they are?
\end{itemize}
When an asymptotic model answer these two questions, it is said to be fully justified. 
In the sequel, after the linear analysis of our model, we refer to section \ref{mainresults} to state the answers of these questions. 
Existence and uniqueness of our solution on a time scale $1/\sqrt{\eps}$ is given by Theorem \ref{localexistence}, while a stability property is provided by Theorem \ref{stability}. 
Finally, the convergence Theorem \ref{convergence} is stated and therefore the full justification of our model is proved.

Let us firstly state some preliminary results in the section below.

\subsection{Properties of the two operators $\Im$ and $\Im^{-1}$.}\label{invert-op}
Assume the nonzero-depth condition that underline the fact that the height of the liquid is always confined, \textit{i.e.} : 
\begin{equation}\label{depthcond}
\exists\quad h_{min} >0, \qquad \inf_{x\in \R} h\ge h_{min} \;\quad  \text{  where  } \; \quad h(t,x)=1+\varepsilon\zeta(t,x) \; .
\end{equation}
Under the above condition, let us introduce the operator $\Im$, where much of the modifications in the previous section  hinges on it, such as:
\begin{equation}\label{op-I}
\Im = h+\eps\mathcal{T}[h]-\eps^2\mathfrak{T} =h-\frac{1}{3}\eps\partial_x( h^3 \partial_x\cdot) +\frac{1}{45}\eps^2\partial_x^4\cdot \; .
\end{equation}
The following lemma states the invertibility results of the operator $\Im$ on well chosen functional spaces.
\begin{lemma}\label{lema1}
Suppose that the depth condition (\ref{depthcond}) is satisfied by the scalar function $\zeta(t,\cdot)\in L^{\infty}(\R)$. Then, the operator
$$
\Im\colon H^4(\R)\longrightarrow L^2(\R)
$$
is well defined, one-to-one and onto .
\end{lemma}
\begin{proof}
We refer to the recent works of two of the authors, \cite[Lemma 1]{KZI2018} and \cite[Lemma 1]{KZI2021},  for the proof of this lemma.
\end{proof}
Some functional properties on the operator $\Im^{-1}$ are given by the Lemma below.
\begin{lemma}\label{lemma2}
Let $t_0>\frac{1}{2}$ and $\zeta\in H^{t_0+1}(\R)$ be such that (\ref{depthcond}) is satisfied. Then, we have the following$\colon$
\begin{enumerate}
\item[(i)] For all $0\leq s\leq t_0+1$, it holds
$$
\vert \Im^{-1}f\vert_{H^s}+\sqrt{\eps}\vert\partial_x \Im^{-1}f\vert_{H^s}+\eps\vert\partial_x^2 \Im^{-1}f\vert_{H^s}\leq C\big(\frac{1}{h_{min}},\vert h-1\vert_{H^{t_0+1}}\big)\vert f\vert_{H^{s}} \; .
$$
and
$$
\sqrt{\eps}\vert \Im^{-1}\partial_xf\vert_{H^s} + \eps\vert \Im^{-1}\partial_x^2f\vert_{H^s}  \leq C\big(\frac{1}{h_{min}},\vert h-1\vert_{H^{t_0+1}}\big)\vert f\vert_{H^{s}} \; .
$$
\item[(iii)] For all $s\geq t_0+1$, it holds
$$
\Vert\Im^{-1}\Vert_{H^s(\R)\rightarrow H^s(\R)}+\sqrt{\eps}\Vert\Im^{-1}\partial_x\Vert_{H^s(\R)\rightarrow H^s(\R)}+\eps\Vert\Im^{-1}\partial_x^2\Vert_{H^s(\R)\rightarrow H^s(\R)}\leq C_s \; ,
$$
and
$$
\sqrt{\eps}\Vert  \Im^{-1}  \partial_x \Vert_{H^s(\R)\rightarrow H^s(\R)}+ \eps \Vert \Im^{-1}\partial_x^2\Vert_{H^s(\R)\rightarrow H^s(\R)}\leq C_s \; ,
$$
\end{enumerate}
where $C_s$ is a constant depending on $1/h_{min}$ , $\vert h-1\vert_{H^s}$ and independent of $\eps\in(0,1)$.
\end{lemma}
\begin{proof}
We refer to the recent works of two of the authors, \cite[Lemma 2]{KZI2018} and \cite[Lemma 2]{KZI2021},  for the proof of this lemma.
\end{proof}
\subsection{Quasilinear form.} \label{quasilinear}
In order to rewrite the extended Boussinesq system  in a condensed form and for the sake of clarity, let us introduce an elliptic  forth-order operator $T[h]$  as follows:
\begin{equation}\label{J}
 T[h] (\cdot)  = h-\eps^2\partial_x^2 (\cdot ) + \frac{2}{45}\eps^2\partial_x^4 (\cdot)   \; .
\end{equation}
The first equation of the system \eqref{boussinesq} can be written as follows:
$$
\partial_t\zeta +\eps v\partial_x\zeta+h\partial_xv = 0 .
$$
Then we apply $\Im^{-1}$ to both sides of the second equation of the system \eqref{boussinesq}, to get:
\begin{equation*}
\partial_tv+\eps vv_x+\Im^{-1}\big(T[h] \zeta_x\big) + \eps^2\Im^{-1}\big(\partial_x^2\zeta_{x}\big)
+\eps^2\Im^{-1}\big(\mathcal{Q}[U]v_x\big) = \mathcal{O} (\eps^3) \; .
\end{equation*}
Hence the higher order Boussinesq system  can be written under the form:
\begin{equation}\label{nonlinear}
\partial_tU+A[U]\partial_xU = 0 \; ,
\end{equation}
where the operator $A$ is denied by:
\begin{equation}\label{AU}
A[U]=\left(
\begin{array}{cc}
\varepsilon v &h\\
\Im^{-1}\big(T[h] \cdot\big) + \eps^2 \Im^{-1}\big(\partial_x^2\cdot\big)& \varepsilon v+\varepsilon^2\Im^{-1}\big(\mathcal{Q}[U]\cdot\big)
\end{array}
\right) \; .
\end{equation}


\subsection{Linear analysis.}\label{linear-analysis}
We consider the following linearized system around a reference state $\underline{U}=(\underline{\zeta},\underline{v})^T$:
\begin{equation}\label{LGN}
	\left\lbrace
	\begin{array}{l}
	\dsp\partial_t U+A[\underline{U}]\partial_x U=0\vspace{1mm},
        \\
	\dsp U_{\vert_{t=0}}=U_0.
	\end{array}\right.
\end{equation}
The energy estimate method needs to  define a suitable energy space for the problem we are considering here. This will permit the convergence of an iterative scheme to construct a solution to the extended Boussinesq 
system \eqref{boussinesq}  for the initial value problem (\ref{LGN}).

\begin{definition}[Energy space]\label{defispace}
 For all $s\ge 0$ and $T>0$, we denote by $X^s$ the vector space $H^{s+2}(\R)\times H^{s+2}(\R)$ endowed with the norm:
\begin{eqnarray*}
\textrm{ for } U=(\zeta,v) \in X^s \,,\, \vert U\vert^2_{X^s}&:=& \vert \zeta\vert^2 _{H^s}+ \eps^2\vert\zeta_x\vert_{H^s}^2+\eps^2\vert \zeta_{xx}\vert^2 _{H^s}+\vert v\vert^2 _{H^s}+\eps\vert v_{x}\vert_{H^{s}}^2+\eps^2\vert v_{xx}\vert_{H^{s}}^2 \; .
\end{eqnarray*}
$X^s_T$ stands for $C([0,\frac{T}{\sqrt{\eps}}];X^{s})$ endowed with its canonical norm.
\end{definition}
\begin{remark}\label{rem2}
It is worth noticing that in the presence of surface tension the second term of the energy norm, $\vert \zeta_x\vert _{H^s}^2$, is controlled by $\eps$ in front of it and this is sufficiently enough to give an existence time scale of order $1/\eps$. 
In fact, the second term here in $\vert \cdot\vert _{X^s}$ is due to the consideration of the vanishing term that is important for Definition \ref{defispace} itself and for controlling higher order terms (see Proposition \ref{prop1}).
\end{remark}
Now we remark that a good suggestion of a pseudo-symmetrizer for $A[\underline{U}]$ requires firstly the introduction of a forth-order linear operator $J[h]$ as follows:
$$
 J[h](\cdot) =1- \eps ^2 \partial_x\big(h^{-1}\partial_x\cdot\big)+\frac{2}{45}\eps^2\partial_x^2\big(h^{-1}\partial_x^2\cdot\big) \; ,
$$
where $\underline{h}=1+\varepsilon\underline{\zeta}$ . Thus a pseudo-symmetrizer for $A[\underline{U}]$ is given by:
\begin{equation}\label{pseudo-symmetrizer}
S=\left( 
\begin{array}{cc}
J[ \underline{h}] & 0 \\\\ 
0& \underline{\Im}
\end{array}
\right)=\left( 
\begin{array}{cc}
1- \eps ^2 \partial_x\big(\underline{h}^{-1}\partial_x\cdot\big)+\frac{2}{45}\eps^2\partial_x^2\big(\underline{h}^{-1}\partial_x^2\cdot\big)  & 0 \\\\ 
0& \underline{h}+\eps \mathcal{T}[\underline{h}]-\eps^2\mathfrak{T}
\end{array}
\right) \; .
\end{equation}
\begin{remark}\label{rem1}
Introducing operator $J[h]$ is of great interest for defining a suitable pseudo-symmetrizer for \eqref{AU}. 
As the higher order derivative in $T[h]$ is not multiplied by $h$ (if this was the case then the vanishing term considered might be $\pm \eps^2h\zeta_{xxx}$), therefore $J[h]$ must replace $T[h]$ in the first entity of \eqref{pseudo-symmetrizer}. This is clearly necessary for controlling $A_2+A_3$ (see Proposition \ref{prop1}).
\end{remark}
Also, a natural energy for the initial value problem (\ref{LGN}) is suggested to be as follows:
\begin{equation}\label{es}
 E^s(U)^2=(\Lambda^sU,S\Lambda^sU) \; .
\end{equation}
\begin{lemma}[Equivalency of $ E^s(U)$ and the $X^s$-norm]\label{lemmaes}
Let $s\geq 0$ and suppose that $ \underline{\zeta}\in L^{\infty}(\R)$ satisfies consition \eqref{depthcond}. Then norm $\vert \cdot\vert_{X^s}$ and the natural energy $E^s(U)$ are uniformly equivalent with respect to $\eps \in (0,1)$ such that:
$$
E^s(U) \leq C\big(h_{min}, \vert\underline{h}\vert_{\infty}\big)\vert U\vert_{X^s} \quad \text{ and } \quad
\vert U \vert_{X^s}\leq C\big(h_{min},\vert \underline{h}\vert_{\infty}\big) E^s(U).
$$
\end{lemma}
\begin{proof}
We refer to the recent work of two of the authors \cite[Lemma 3]{KZI2018} for the proof of this important property.
\end{proof}
The well-posedness and a derivation of a first energy estimate for the linear system is given in the following proposition.
\begin{proposition}[Well-posedness \& energy estimate of the linear system]\label{prop1}
For $t_0>\frac{1}{2}$, $s\geq t_0+1$ and under the depth condition (\ref{depthcond}), suppose that $\underline{U}=(\underline{\zeta}, \underline{v})^T$ $\in X^{s}_{T}$
 and $\partial_t \underline{U} \in X^{s-1}_{T}$ 
 at any time in $[0,\frac{T}{\sqrt{\varepsilon}}]$. Then, there exists a unique solution $U=(\zeta, v)^T$ $\in X^{s}_{T} $ to (\ref{LGN}) for any initial data $U_0$ in $X^s$ and for all $0\leq t\leq\frac{T}{\sqrt{\varepsilon}}$ it holds that:
\begin{equation}\label{energy}
\dsp E^s\big(U(t)\big)\displaystyle\leq \big(e^{\sqrt{\varepsilon}\lambda_{T} t}\big)^{1/2}E^s(U_0) \; ,
\end{equation}
for some $\lambda_{T}$ depending only on $ h_{min}^{-1},  \sup_{0\leq \sqrt{\eps} t\leq T}E^s(\underline{U}(t))$ and $\sup_{0\leq \sqrt{\eps}  t \leq T}\vert\partial_t\underline{h}(t) \vert_{L^{\infty}}$ .
\end{proposition}
\begin{proof}
For the proof of the existence and uniqueness of the solution, we refer to the proof found in \cite[Appendix A]{Israwi2011} which can be directly adapted to the problem we are considering here. 

Thereafter, we will focus our attention on the proof of the energy estimate \eqref{energy}. First of all, fix $\lambda\in\R$. The proof of the energy estimate is centered on bounding from above by zero the expression
$
e^{\sqrt{\varepsilon}\lambda t}\partial_t(e^{-\sqrt{\varepsilon}\lambda t}E^s(U)^2).
$
For this sake, we use the fact that $\underline{\Im}$ and $J[\underline{h}]$ are symmetric to evaluate the expression under the form:
\begin{align*}
\frac{1}{2}e^{\sqrt{\varepsilon}\lambda t}\partial_t(e^{-\sqrt{\varepsilon}\lambda t}E^s(U)^2)&=-\frac{\lambda}{2} \sqrt{\varepsilon} E^s(U)^2 -\big(SA[\underline{U}]\Lambda^s\partial_x U,\Lambda^s U\big)- \big(\big[\Lambda^s,A[\underline{U}]\big] \partial_xU,S\Lambda^s U\big)\\
&\quad+\frac{1}{2}\big(\Lambda^s\zeta,[\partial_t,J[\underline{h}]]\Lambda^s\zeta\big)+\frac{1}{2}(\Lambda^sv,[\partial_t,\underline{\Im}]\Lambda^sv) \; .
\end{align*}
Now it remains to control the r.h.s components of the above equation. To do so, we firstly recall the commutator estimate we shall use due to Kato-Ponce \cite{KP88} and recently improved by Lannes \cite{Lannes2006}: in particular, for any $s>3/2$, and $ q \in H^s(\R),p\in H^{s-1}(\R)$, one has:
\begin{equation}\label{cest}
\big\vert [\Lambda^s, q]p\vert_{2} \lesssim \vert \nabla q\vert_{H^{s-1}}\vert p\vert_{H^{s-1}} \; .
\end{equation}
Also we shall use intensively the classical product estimate (see \cite{AG91,Lannes2006,KP88}): in particular, for any $p,q\in H^s(\R^2)$, $s>3/2$, one has:
\begin{equation}\label{mest}
\vert pq\vert_{H^s}\lesssim \vert q\vert_{H^s}\vert p\vert_{H^s} \; .
\end{equation}
$\bullet$  Estimation of $(SA[\underline{U}]\Lambda^s\partial_x U,\Lambda^s U).$ We have:
\begin{equation*}
SA[\underline{U}]=\left( 
\begin{array}{cc}
 \eps J[\underline{h}](\underline{v}\cdot)& J[\underline{h}](\underline{h}\cdot) \vspace{1mm} \\
T[\underline{h}] \cdot  +\eps^2 \partial_x^2 \cdot \hspace{ 1mm} & \eps\underline{\Im}(\underline{v}\cdot)+\eps^2\mathcal{Q}[\underline{U}]\cdot 
\end{array}
\right),
\end{equation*}
then it holds that:
\begin{align*}
\big(SA[\underline{U}]\Lambda^s\partial_x U,\Lambda^s U\big) &=\eps\big(J[\underline{h}](\underline{v}\Lambda^s\zeta_x) , \Lambda^s\zeta\big)+\big(J[\underline{h}](\underline{h}\Lambda^sv_x),\Lambda^s\zeta\big)+\big(T[\underline{h}]\Lambda^s\zeta_x , \Lambda^sv\big) \\
& +\eps^{2}\big(\Lambda^s\zeta_{xxx},\Lambda^sv\big)+\eps\big(\underline{\Im}(\underline{v}\Lambda^sv_x),\Lambda^sv\big) 
+\eps^2\big(\mathcal{Q}[\underline{U}]\Lambda^sv_x,\Lambda^sv\big) =A_1+A_2+...+A_{6} \; .
\end{align*}
To control $A_1$, by integration by parts, we have:
\begin{align*}
A_1&=\eps\big(\underline{v}\Lambda^s\zeta_x,\Lambda^s\zeta\big)+\eps^3 \big( \underline{h}^{-1} \partial_x(\underline{v}\Lambda^s\zeta_x),\Lambda^s\zeta_x\big)+\frac{2}{45}\eps^3\big( \underline{h}^{-1} \partial_x^2(\underline{v}\Lambda^s\zeta_x),\Lambda^s\zeta_{xx}\big)=A_{11}+A_{12}+A_{13} \; .
\end{align*}
Clearly, it holds that:
\begin{align*}
\vert A_{11}\vert&= \frac{1}{2}\eps\vert\big(\Lambda^s\zeta,\underline{v}_x\Lambda^s\zeta\big)\vert\leq \eps C\big(\vert\underline{v}\vert_{W^{1,\infty}}\big)E^s(U)^2.
\end{align*}
By integrating by parts, it holds that:
\begin{align*}
\vert A_{12}\vert= \eps^3 \big(\underline{h}^{-1} \underline{v}_x\Lambda^s\zeta_{x},\Lambda^s\zeta_x\big) +  \eps^3 \big(\underline{h}^{-1} \underline{v}\Lambda^s\zeta_{xx},\Lambda^s\zeta_x\big) \leq \eps C\big(h_{min}^{-1} , \vert\underline{v}_x\vert_{\infty}\big)E^s(U)^2.
\end{align*}
Now using the fact that:
\begin{equation}\label{deriv}
\partial_x^2(MN)=N\partial_x^2M+2M_xN_x+M\partial_x^2N \; ,
\end{equation}
for any differentiable functions $M$, $N$ and by integration by parts, we have:
\begin{align*}
 A_{13} &=\frac{2}{45}\eps^3\big[\big( \underline{h}^{-1} \underline{v}_{xx}\Lambda^s\zeta_x,\Lambda^s\zeta_{xx}\big)+2\big(\underline{h}^{-1}\underline{v}_x\Lambda^s\zeta_{xx},\Lambda^s\zeta_{xx}\big)
 +\frac{1}{2}\big(\underline{h}^{-2}\underline{h}_x\underline{v}\Lambda^s\zeta_{xx},\Lambda^s\zeta_{xx}\big)-\frac{1}{2}\big(\underline{h}^{-1}\underline{v}_x\Lambda^s\zeta_{xx},\Lambda^s\zeta_{xx}\big)\big]\\
 &=A_{131}+...+A_{314} \; .
 \end{align*}
Although $A_{131}$ can be controlled directly with $\sqrt{\eps}$ in front of the constant, one may improve this by $\eps$ instead. Indeed by integration by parts one has:
$$
A_{131}= \frac{2}{45}\eps^3\big(\underline{h}^{-2}\underline{h}_x\underline{v}_{xx}\Lambda^s\zeta_x,\Lambda^s\zeta_{x}\big)
-\frac{2}{45}\eps^3\big(\underline{h}^{-1}\underline{v}_{xxx}\Lambda^s\zeta_x,\Lambda^s\zeta_{x}\big)=A_{1311}+A_{1312}.
$$
Remark that $\underline{h}_x=\eps\underline{\zeta}_x$, then $A_{1311}$ posses sufficient $\eps$'s, unlike $A_{1312}$ on which we have to work a little more. Indeed, in view of \eqref{depthcond} we have that $\underline{h}^{-1}>0$, then it holds:
\begin{align*}
A_{1312}=-\frac{2}{45}\eps^3\big(\underline{h}^{-1}\underline{v}_{xxx},(\Lambda^s\zeta_x)^2\big)\leq \frac{2}{45}\eps^3\vert \underline{v}_{xxx}\vert_{\infty}\big(\underline{h}^{-1},(\Lambda^s\zeta_x)^2\big).
\end{align*}
Again by integration by parts, we get :$\big(\underline{h}^{-1},(\Lambda^s\zeta_x)^2\big)= (\underline{h}^{-2}\underline{h} _x\Lambda^s\zeta,\Lambda^s\zeta_x)-(\underline{h}^{-1}\Lambda^s\zeta,\Lambda^s\zeta_{xx})$. 
Therefore one may control $A_{1312}$ by $\eps C(h_{min}^{-2}, \vert\zeta\vert_{W^{1,\infty}},\mu\vert \underline{v}_{xxx}\vert_{\infty})E^s(U)^2$. Consequently, it holds:
$$
A_{1311}+A_{132}+..+A_{134}\leq \eps C\big(h_{min}^{-2}, \vert\zeta\vert_{W^{1,\infty}},\vert\underline{v}\vert_{W^{1,\infty}},\sqrt{\eps}\vert\underline{v}_{xx}\vert_{\infty}\big)  E^s(U)^2.
$$
Collecting the information provided above we get:
$$
\vert A_1\vert\leq \eps C\big(h_{min}^{-2}, \vert\underline{\zeta}\vert_{W^{1,\infty}},\vert\underline{v}\vert_{W^{1,\infty}},\sqrt{\eps}\vert\underline{v}_{xx}\vert_{\infty}\big)  E^s(U)^2 \; .
$$
To control $A_2+ A_3$, by remarking firstly that $J[\underline{h}]$ and $T[\underline{h}]$ are symmetric, and then by integration by parts after having performing some algebraic calculations and using (\ref{deriv}), we have:
\begin{equation*}
 A_2+A_3= -\big(\Lambda^sv,\underline{h}_x\Lambda^s\zeta\big)
 +\eps^2\big(\underline{h}^{-1} \underline{h}_{x} \Lambda^s\zeta_{x},\Lambda^sv_{x}\big) 
+\frac{4}{45}\eps^2\big(\underline{h}^{-1}\underline{h}_{x}\Lambda^sv_{xx},\Lambda^s\zeta_{xx}\big)
-\frac{2}{45}\eps^2\big( \underline{h}^{-1}\underline{h}_{xx} \Lambda^s\zeta_{xx},\Lambda^sv_{x}\big) \; .
\end{equation*}
Unfortunately, an inconvenient term appears in $A_2+A_3$: it is the term $\eps^2\big(\underline{h}^{-1}\underline{h}_{xx}\Lambda^s\zeta_{xx},\Lambda^sv_x\big)$. 
This term won't be controlled without gaining $\sqrt{\eps}$ taken from $\underline{h}_{xx}=\eps\underline{\zeta}_{xx}$ and the other $\sqrt{\eps}$ sits in front of the constant. Due to this fact, it follows that:
 $$
\vert A_2+A_3\vert\le  \sqrt{\eps} C\big(h_{min}^{-1},\vert\underline{\zeta}\vert_{W^{1,\infty}},\vert\underline{v}\vert_{W^{1,\infty}},\eps\vert\underline{\zeta}_{xx}\vert_{H^s}\big)  E^s(U)^2.
$$
To control $A_4$, by integration by parts, it holds:
\begin{align*}
A_4 =- \eps^2 (\Lambda^s\zeta_{xx},\Lambda^sv_x) \leq \sqrt{\eps}  E^s(U)^2 \; .
\end{align*}
To control $A_5$, by integration by parts, we have:
\begin{align*}
A_5&=\eps\big(\underline{h}\underline{v}\Lambda^sv_x,\Lambda^sv\big)+\frac{\eps^2}{3}\big( \underline{h} ^3 \partial_x(\underline{v}\Lambda^sv_x),\Lambda^sv_x\big)+\frac{\eps^3}{45}\big(\partial_x^2(\underline{v}\Lambda^sv_x),\Lambda^sv_{xx}\big)=A_{51}+A_{52}+A_{53}
\end{align*}
where
$$
\big\vert A_{51}\big\vert=\big\vert-\frac{\eps}{2}\big(\underline{h}_x\underline{v}\Lambda^sv,\Lambda^sv\big)-\frac{\eps}{2}\big(\underline{h}\underline{v}_x\Lambda^sv,\Lambda^sv\big)\big\vert\leq \eps C\big(\vert\underline{\zeta}_x\vert_{\infty},\vert\underline{v}_x\vert_{\infty}\big)E^s(U)^2
$$
with
$$
\big\vert A_{52}\big\vert=\big\vert-\frac{\eps^2}{2}\big(\underline{h}_x^3\underline{v}\Lambda^sv_x,\Lambda^sv_x\big)-\frac{\eps^2}{6}\big(\underline{h}^3\underline{v}\Lambda^sv_x,\Lambda^sv_x\big)\big\vert\leq\eps C\big(\vert\underline{\zeta}\vert_{W^{1,\infty}}\big)E^s(U)^2
$$
and
\begin{align*}
\big\vert A_{53}\big\vert&=\frac{\eps^2}{45}\big\vert\big(\underline{v}_{xx}\Lambda^sv_x,\Lambda^sv_{xx}\big)+2\big(\underline{v}_x\Lambda^sv_{xx},\Lambda^sv_{xx}\big)-\frac{1}{2}\big(\underline{v}_x\Lambda^sv_{xx},\Lambda^sv_{xx}\big) \big\vert \leq \eps C\big(\vert\underline{\zeta}\vert_{W^{1,\infty}},\sqrt{\eps}\vert\underline{v}_{xx}\vert_{\infty}\big)E^s(U)^2.
\end{align*}
Therefore, it holds that:
$$
\vert A_{5}\vert\leq\eps C\big(\vert\underline{\zeta}\vert_{W^{1,\infty}},\vert\underline{v}_x\vert_{\infty},\sqrt{\eps}\vert\underline{v}_{xx}\vert_{\infty}\big)E^s(U)^2.
$$
Finally, by integration by parts, $A_6$ is controlled by $\eps C\big( \vert\underline{v}_x\vert_{\infty}\big)E^s(U)^2$. Therefore, it holds:
$$
\big\vert \big(SA[\underline{U}]\Lambda^s\partial_x U,\Lambda^s U\big)\big\vert\leq \sqrt{\eps} C\big(\vert\zeta\vert_{W^{1,\infty}},\eps\vert\underline{\zeta}_{xx}\vert_{H^s},\vert\underline{v}\vert_{W^{1,\infty}},\sqrt{\eps}\vert\underline{v}_{xx}\vert_{\infty}\big)  E^s(U)^2 \; .
$$
$\bullet$ Estimation of $\big(\big[\Lambda^s,A[\underline{U}]\big]\partial_xU,S\Lambda^sU\big)$. Let us remark that:
\begin{align*}
\big(\big[\Lambda^s,A[\underline{U}]\big]\partial_xU,S\Lambda^sU\big)&=\eps\big([\Lambda^s,\underline{v}]\zeta_x,J[\underline{h}]\Lambda^s\zeta\big)
+ \big([\Lambda^s,\underline{h}]v_x,J[\underline{h}]\Lambda^s\zeta\big)
+\big([\Lambda^s,\underline{\Im}^{-1}(T[\underline{h}]\cdot)]\zeta_x,\underline{\Im}\Lambda^sv\big)\\
& + \eps^2 \big([\Lambda^s,\underline{\Im}^{-1}(\partial_x^2\cdot)]\zeta_x,\underline{\Im}\Lambda^sv\big)+\eps\big([\Lambda^s,\underline{v}]v_x,\underline{\Im}\Lambda^sv\big)+\eps^2\big([\Lambda^s,\underline{\Im}^{-1}(\mathcal{Q}[\underline{U}]\cdot)]v_x,\underline{\Im}\Lambda^sv\big) \\&  =B_1+B_2+...+B_{6}.
\end{align*}
To control $B_1$, we use the expression of $J[\underline{h}]$ to write:
$$
B_1 =\eps\big([\Lambda^s,\underline{v}]\zeta_x, \Lambda^s\zeta\big) + \eps^3\big(\partial_x[\Lambda^s,\underline{v}]\zeta_x, \frac{1}{\underline{h}} \Lambda^s\zeta_{x}\big)
+\frac{2}{45}\eps^3\big(\partial_x^2[\Lambda^s,\underline{v}]\zeta_x, \underline{h}^{-1} \Lambda^s\zeta_{xx}\big) \; .
$$
Then by using the fact that:
\begin{equation}\label{MN}
\partial_x[\Lambda^s,M]N=[\Lambda^s,M_x]N + [\Lambda^s,M]N_x \; \text{ and }\; \partial_x^2[\Lambda^s,M]N = [\Lambda^s,M_{xx}]N+2[\Lambda^s,M_x]N_x+[\Lambda^s,M]N_{xx} \; ,
\end{equation}
and using \eqref{cest}, it holds that:
\begin{align*}
B_1 & =\eps\big([\Lambda^s,\underline{v}]\zeta_x, \Lambda^s\zeta\big)
+ \eps^3\big([\Lambda^s,\underline{v}_x]\zeta_x, \underline{h}^{-1} \Lambda^s\zeta_{x}\big)
+ \eps^3\big([\Lambda^s,\underline{v}]\zeta_{xx}, \underline{h}^{-1} \Lambda^s\zeta_{x}\big)\\
& \quad+\frac{2}{45}\eps^3 \Big\lbrace\big([\Lambda^s,\underline{v}_{xx}]\zeta_x, \underline{h}^{-1} \Lambda^s\zeta_{xx}\big)
+ 2\big([\Lambda^s,\underline{v}_x]\zeta_{xx}, \underline{h}^{-1} \Lambda^s\zeta_{xx}\big)
+\big([\Lambda^s,\underline{v}]\zeta_{xxx},\underline{h}^{-1} \Lambda^s\zeta_{xx}\big)\Big\rbrace\\
&\leq \sqrt{\eps} C\big(h_{min}^{-1} ,\vert\underline{v}\vert_{H^s},\eps\vert\underline{v}_{xx}\vert_{H^s}\big)E^s(U)^2 \; .
\end{align*}
The $\sqrt{\eps}$ in front of the constant is due to the inconvenient term represented by $\eps^3 \big([\Lambda^s,\underline{v}_x]\zeta_{xx},\underline{h}^{-1}\Lambda^s\zeta_{xx}\big)$. \\
To control $B_2$, by the expression of  $J[\underline{h}]$ and \eqref{MN}, we have:
\begin{multline*}
 B_2= \big([\Lambda^s,\underline{h}-1]v_x, \Lambda^s\zeta\big)
 + \eps^3 \big([\Lambda^s,\underline{\zeta}_x]v_x, \underline{h}^{-1} \Lambda^s\zeta_{x}\big)
  + \eps^2 \big([\Lambda^s,\underline{h}-1]v_{xx}, \underline{h}^{-1} \Lambda^s\zeta_{x}\big)\\
 +\frac{2}{45}\eps^2\Big\lbrace\big([\Lambda^s,(\underline{h}-1)_{xx}]v_x,  \underline{h}^{-1} \Lambda^s\zeta_{xx}\big)
+2\big([\Lambda^s,(\underline{h}-1)_{x}]v_{xx},  \underline{h}^{-1}\Lambda^s\zeta_{xx}\big)
+\big([\Lambda^s,\underline{h}-1]v_{xxx}, \underline{h}^{-1}\Lambda^s\zeta_{xx}\big)\Big\rbrace.
\end{multline*}
Then, clearly the following estimate holds:
$$
\vert B_2\vert\leq \eps C\big(h_{min}^{-1} ,\vert\underline{h}-1\vert_{H^s},\eps\vert\underline{\zeta}_{xx}\vert_{H^s}\big)E^s(U)^2.
$$
To control $B_3$, we have that $\underline{\Im}$ is symmetric and that:
\begin{equation*}
\underline{\Im}[\Lambda^s, \underline{\Im}^{-1}]T[\underline{h}]\zeta_x=\underline{\Im}[\Lambda^s, \underline{\Im}^{-1}T[(\underline{h}]\cdot)]\zeta_x-[\Lambda^s, T[\underline{h}]]\zeta_x \; .
\end{equation*}
Moreover, since $[\Lambda^s,\underline{\Im}^{-1}]=-\underline{\Im}^{-1}[\Lambda^s,\underline{\Im}]\underline{\Im}^{-1}$, one gets:
\begin{equation*}
\underline{\Im}[\Lambda^s, \underline{\Im}^{-1}\;T[\underline{h}] \cdot ]\zeta_x=-[\Lambda^s, \underline{\Im}] \underline{\Im}^{-1}T[\underline{h}]\zeta_x+[\Lambda^s, T[\underline{h}] ]\zeta_x \; .
\end{equation*}
Therefore, one may write:
\begin{align*}
B_3 &=\big([\Lambda^s,\underline{\Im}]\underline{\Im}^{-1}(T[\underline{h}]\zeta_x),\Lambda^sv\big) +\big([\Lambda^s,T[\underline{h}]]\zeta_x,\Lambda^sv\big)   \; .
\end{align*}
At this point, using the expressions of $T[\underline{h}]$ and $J[\underline{h}]$, it holds:
$$
\frac{2}{45}\eps^2 \partial_x^4\zeta_x = 2\underline{\Im}\zeta_x-2\underline{h}\zeta_x+\frac{2}{3} \eps \partial_x(\underline{h}^3\zeta_{xx}) \;  .
$$
Therefore, it holds that:
$$
\underline{\Im}^{-1}(T[\underline{h}]\zeta_x)=2\zeta_x - \underline{\Im}^{-1}(\underline{h}\zeta_x)- \eps^2\underline{\Im}^{-1}( \zeta_{xxx})+\frac{2}{3}\eps\underline{\Im}^{-1}\partial_x(\underline{h}^3\zeta_{xx}) \; ,
$$
which implies that:
\begin{align*}
B_3 &=2\big([\Lambda^s,\underline{\Im}]\zeta_x,\Lambda^sv\big) -\big([\Lambda^s,\underline{\Im}]\underline{\Im}^{-1}(\underline{h}\zeta_x),\Lambda^s v\big)+\frac{2}{3}\eps\big([\Lambda^s,\underline{\Im}]\underline{\Im}^{-1}\partial_x(\underline{h}_3\zeta_{xx}),\Lambda^sv\big)\\
& \quad -\eps^2 \big([\Lambda^s,\underline{\Im}]\underline{\Im}^{-1}( \zeta_{xxx}),\Lambda^sv\big) +\big([\Lambda^s,T[\underline{h}]]\zeta_x,\Lambda^sv\big)\\
&=B_{31}+B_{32}+B_{33}+B_{34}+B_{35}.
\end{align*}
Thanks to the fact that, for all $k\in\N, \underline{h}^{k}-1=\mathcal{O}(\eps\underline{\zeta})$ and using the explicit expression of $\underline{\Im}$ combined with the identities: 
\begin{equation}\label{commu}
[\Lambda^s,\partial_x(M\partial_x\cdot)]N = \partial_x[\Lambda^s,M]N_x \qquad\text{ and }\qquad [\Lambda^s, \partial_x^m]N=0 \; \quad\forall \; m\in\N^* \; ,
\end{equation}
 then by integration by parts and \eqref{cest}, it holds that:
$$
B_{31} = 2 \big([\Lambda^s,\underline{h}-1]\zeta_x,\Lambda^sv\big)+\frac{2}{3}\eps\big([\Lambda^s,\underline{h}^3-1]\zeta_{xx},\Lambda^sv_x\big)
\le \sqrt{\eps} C\big(\vert\underline{h}-1\vert_{H^s})E^s(U)^2 \; .
$$
Also, by \eqref{cest} it holds:
\begin{multline*}
\vert B_{32}\vert\leq\big\vert \big([\Lambda^s,\underline{h}]\underline{\Im}^{-1}(\underline{h}\zeta_x),\Lambda^sv\big)+\frac{1}{3}\eps\big([\Lambda^s,\underline{h}^3]\partial_x\underline{\Im}^{-1}(\underline{h}\zeta_x),\Lambda^sv_x\big) \big\vert\leq \eps C\big(\vert\underline{h}-1\vert_{H^s},C_s)E^s(U)^2 \; ,
\end{multline*}
with
\begin{multline*}
\vert B_{33}\vert\leq\big\vert \frac{2}{3}\eps\big([\Lambda^s,\underline{h}]\underline{\Im}^{-1}\partial_x(\underline{h}^3\zeta_{xx}),\Lambda^sv\big)+\frac{2}{9}\eps^2\big([\Lambda^s,\underline{h}^3]\partial_x\underline{\Im}^{-1}\partial_x(\underline{h}^3\zeta_{xx}),\Lambda^sv_x\big)  \big\vert\leq\eps C\big(\vert\underline{h}-1\vert_{H^s},C_s)E^s(U)^2 \; ,
\end{multline*}
and
\begin{multline*}
\vert B_{34}\vert\leq \eps^2\big\vert \big([\Lambda^s,\underline{h}]\underline{\Im}^{-1}(\zeta_{xxx}),\Lambda^sv\big)+\frac{1}{3}\eps^3\big([\Lambda^s,\underline{h}^3]\partial_x\underline{\Im}^{-1}(\zeta_{xxx}),\Lambda^sv_x\big) \big\vert\leq\eps C\big(\vert\underline{h}-1\vert_{H^s},C_s)E^s(U)^2 \; .
\end{multline*}
For controlling $B_{35}$, the explicit expression of $T[\underline{h}]$ and \eqref{commu} gives that:
\begin{equation*}
B_{35}=\big([\Lambda^s,\underline{h}-1]\zeta_x,\Lambda^sv\big)  \le \eps C\big(\vert\underline{h}-1\vert_{H^s},C_s)E^s(U)^2 \; .
\end{equation*}
Thus, as a conclusion, it holds that:
$$
\vert B_3\vert\leq \sqrt{\eps}  C\big(\vert\underline{h}-1\vert_{H^s},\vert\underline{\zeta}\vert_{\infty},\eps\vert\underline{\zeta}_{xxx}\vert_{H^{s-1}},C_s\big)E^s(U)^2.
$$
To control $B_4$, as for $B_3$ and using \eqref{commu} one may write:
\begin{align*}
B_4 &= -\eps^2\big([\Lambda^s, \underline{\Im}]\underline{\Im}^{-1}\zeta_{xxx},\Lambda^s v\big)\\
& = -\eps^2\big([\Lambda^s, \underline{h} ]\underline{\Im}^{-1}\zeta_{xxx},\Lambda^s v\big) 
-\frac{1}{3}\eps^3\big([\Lambda^s, \underline{h}^3]\partial_x\underline{\Im}^{-1}\zeta_{xxx},\Lambda^s v_x\big) 
\leq\eps C\big(\vert\underline{h}-1\vert_{H^s},C_s)E^s(U)^2 \; .
\end{align*}
To control $B_5$, using the expression of $\underline{\Im}$, \eqref{cest} and (\ref{MN}) with integration by parts and the fact that $\partial_x[\Lambda^s,M]N=[\Lambda^s,M_x]N+[\Lambda^s,M]N_x$, it holds:
\begin{align*}
&\vert B_5\vert=\eps\big\vert\big([\Lambda^s,\underline{v}]v_x,\underline{h}\Lambda^sv\big)+\frac{1}{3}\eps\big([\Lambda^s,\underline{v}_x]v_x,\underline{h}^3\Lambda^sv_x\big)+\frac{1}{3}\eps\big([\Lambda^s,\underline{v}]v_{xx},\underline{h}^3\Lambda^sv_x\big)+\frac{1}{45}\eps^2\big([\Lambda^s,\underline{v}_{xx}]v_x, \Lambda^sv_{xx}\big)\\
&+\frac{2}{45}\eps^2\big([\Lambda^s,\underline{v}_x]v_{xx}, \Lambda^sv_{xx}\big)+\frac{1}{45}\eps^2\big([\Lambda^s,\underline{v}]v_{xxx}, \Lambda^sv_{xx}\big)\big\vert\leq\eps C\big(\vert\underline{h}\vert_{\infty},\vert\underline{v}\vert_{H^s},\sqrt{\eps}\vert\underline{v}_{xx}\vert_{H^{s-1}},\eps\vert\underline{v}_{xxx}\vert_{H^{s-1}}\big)E^s(U)^2.
\end{align*}
To control $B_6$, using the same arguments as the ones used to control $B_3$, using expression of $\underline{\Im}$, \eqref{cest} and \eqref{commu}, it follows that:
\begin{equation*}
B_6=-\eps^2\big([\Lambda^s,\underline{h}]\underline{\Im}^{-1}\mathcal{Q}[\underline{U}]v_x,\Lambda^sv\big)-\frac{\eps^3}{3}\big([\Lambda^s,\underline{h}^3]\partial_x\underline{\Im}^{-1}\mathcal{Q}[\underline{U}]v_x,\Lambda^sv_x\big) +\eps^2\big([\Lambda^s,\mathcal{Q}[\underline{U}]]v_x,\Lambda^sv\big).
\end{equation*}
Now, using  the expression of $\mathcal{Q}$ with the help of Lemma \ref{lemma2}, estimate \eqref{cest}, in addition to (\ref{commu}) and the fact that $[\Lambda^s, \partial_x(M\cdot)]N= \partial_x[\Lambda^s, M]N$, it holds:
$$
\vert B_{6}\vert\leq \eps C\big(\vert\underline{h}-1\vert_{H^s}, \sqrt{\eps}\vert\underline{v}_x\vert_{H^{s}},C_s\big)E^s(U)^2 \; .
$$
Eventually, as a conclusion, one gets:
$$
\big\vert\big(\big[\Lambda^s,A[\underline{U}]\big]\partial_xU,S\Lambda^sU\big)\big\vert\leq \sqrt{\eps} C\big( h_{min}^{-1}, \vert\underline{h}-1\vert_{H^s},\vert\underline{\zeta}\vert_{H^{s}},\eps\vert\underline{\zeta}_{xx}\vert_{H^s},\vert\underline{v}\vert_{H^{s}},\sqrt{\eps}\vert\underline{v}_x\vert_{H^s},\eps\vert\underline{v}_{xx}\vert_{H^s},C_s\big)E^s(U)^2.
$$
It is worth noticing that $\sqrt{\eps}$ in front of the constant is due to $B_1$ and $B_{31}$. \\
$\bullet$ Estimation of $\big(\Lambda^s\zeta,[\partial_t, J[\underline{h}]]\Lambda^s\zeta\big)$. Using the expression of $J[\underline{h}]$ and by integration by parts, it holds that:
$$
\big(\Lambda^s\zeta,[\partial_t, J[\underline{h}]]\Lambda^s\zeta\big)\big\vert=  \eps^2\big(\underline{h}^{-2} \partial_t\underline{h} \Lambda^s\zeta_{x},\Lambda^s\zeta_{x}\big) + \frac{2}{45}\eps^2\big(\underline{h}^{-2} \partial_t\underline{h} \Lambda^s\zeta_{xx},\Lambda^s\zeta_{xx}\big)  \leq \eps C(h_{min}^{-2} , \vert\partial_t\underline{\zeta}\vert_{\infty})E^s(U)^2 .
$$
$\bullet$ Estimation of $\big(\Lambda^sv,[\partial_t,\underline{\Im}]\Lambda^sv\big)$. It holds that:
$$
 [\partial_t , \underline{h}]\Lambda^sv =  \partial_t \underline{h}\Lambda^sv  \qquad\text{and }\qquad
 [\partial_t , \partial_x(\underline{h}^3\partial_x\cdot)]\Lambda^sv = \partial_x(\partial_t \underline{h}^3\Lambda^sv_x) \; ,
$$
then by integration by parts: 
\begin{equation*}
\big\vert\big(\Lambda^sv,[\partial_t,\underline{\Im}]\Lambda^sv\big)\big\vert=\big\vert\big(\partial_t\underline{h}\Lambda^sv,\Lambda^sv\big)+\frac{\eps}{3}\big(\partial_t\underline{h}^3\Lambda^sv_x,\Lambda^sv_x\big) \big\vert\leq \eps C(\vert\partial_t\underline{\zeta}\vert_{\infty},E^s(\underline{U}))E^s(U)^2.
\end{equation*}
Finally, combining the above estimates in addition to that fact that $H^s(\R)$ is continuously embedded in $W^{1,\infty}(\R)$, it holds that:
$$
\frac{1}{2}e^{\sqrt{\varepsilon}\lambda t}\partial_t (e^{-\sqrt{\varepsilon}\lambda t}E^s(U)^2) \leq \sqrt{\varepsilon}\big(C(h_{min}^{-1},E^s(\underline{U}))-\lambda\big)E^s(U)^2.
$$
Taking $\lambda=\lambda_T$ large enough (how large depending on 
$\displaystyle \sup_{t\in [0,\frac{T}{\sqrt{\varepsilon}}]}C(h_{min}^{-1},E^s(\underline{U}))$
such that the right hand side of the inequality above is negative for all $t\in [0,\frac{T}{\sqrt{\varepsilon}}]$, then it holds that:
$$
\forall\hspace{0.1cm}t\in \Big[0,\frac{T}{\sqrt{\varepsilon}}\Big ]\hspace{0.1cm},\quad\qquad \frac{1}{2}e^{\sqrt{\varepsilon}\lambda t}\partial_t \big(e^{-\sqrt{\varepsilon}\lambda t}E^s(U)^2\big) \leq0.
$$
Thanks to Gr$\ddot{\text{o}}$nwall's inequality so that it holds
\begin{equation*}
\forall\hspace{0.1cm}t\in \Big[0,\frac{T}{\sqrt{\varepsilon}}\Big]\hspace{0.1cm},\quad\qquad
E^s\big(U(t)\big)\displaystyle\leq \big(e^{\sqrt{\varepsilon}\lambda_{T} t}\big)^{1/2}E^s(U_0) \;, 
\end{equation*}
and hence the desired energy estimate is finally obtained.
\end{proof}

\subsection{Main results.}\label{mainresults}

\subsubsection{Well-posedness of the extended Boussinesq system.}
Theorem \ref{localexistence} represents the well-posedness of the extended Boussinesq system \eqref{boussinesq} which holds in $X^s=H^{s+2}(\R)\times H^{s+2}(\R)$ as soon as $s>3/2$ on a time interval of size $1/\sqrt{\eps}$.
\begin{theorem}[Local existence]\label{localexistence}
Suppose that $U_0=(\zeta_0,v_0)\in X^s$ satisfying (\ref{depthcond}) for any $t_0>\frac{1}{2}$, $s\geq t_0+1$. Then there exists a maximal time $T_{max}=T(\vert U_0\vert_{X^s})>0$ and a unique solution $U=(\zeta,v)^T\in X^s_{T_{max}}$ to the extended Boussinesq system \eqref{boussinesq} with initial condition $(\zeta_0,v_0)$ such that the non-vanishing depth condition (\ref{depthcond}) is satisfied for any $t\in [0,\frac{T_{max}}{\sqrt{\varepsilon}})$.
In particular if $T_{max}<\infty$ one has
$$ \vert U(t,\cdot)\vert_{X^s}\longrightarrow\infty\quad\hbox{as}\quad t\longrightarrow \frac{T_{max}}{\sqrt{\varepsilon}},\qquad\text{ or } \qquad
 \inf_{\R} h(t,\cdot)=\inf_{\R}1+\varepsilon\zeta(t,\cdot)\longrightarrow 0 \quad\hbox{as}\quad t\longrightarrow \frac{T_{max}}{\sqrt{\varepsilon}} \; .
$$
\end{theorem}
\begin{proof}
The proof follows same line as  \cite[Theorem 1]{KZI2018} using the energy estimate proved in Proposition \ref{prop1}. This is due to the fact that in \cite{KZI2018} a most general case is considered (\textit{i.e.} the extended Green-Naghdi equations). Remark that the proof itself is an adaptation of the proof of the well-posedness of hyperbolic systems (see \cite{AG91} for general details). 
\end{proof}

\subsubsection{A stability property.}
Theorem \ref{localexistence} is complemented by the following result that shows the stability of the solution with respect to perturbations, which is very useful for the justification of asymptotic approximations of the exact solution. (The solution $U=(\zeta,v)^{T}$ and time $T_{max}$ that appear in the statement below are those furnished by Theorem \ref{localexistence}).
\begin{theorem}[Stability]\label{stability}
Suppose that the assumption of Theorem \ref{localexistence} is satisfied and moreover assume that there exists $\widetilde{U}=(\widetilde{\zeta},\widetilde{v})^{T}\in C\left([0,\frac{T_{max}}{\sqrt{\eps}}], X^{s+1}(\R)\right)$ such that 
\begin{equation*}
\left\{
\begin{array}{lcl}
\displaystyle\partial_t\widetilde{\zeta}+\partial_x(\widetilde{h}\widetilde{v})=f_1\vspace{1mm},\\
\displaystyle \tilde{\Im} \big(\partial_t\tilde{ v}+\varepsilon \tilde{v}\tilde{v}_x\big) + \tilde{h}\partial_x\tilde{\zeta} - \eps^{2} \tilde{\zeta}_{xxx}+\frac{2}{45}\eps^2 \tilde{\zeta}_{xxxxx}    + \eps^{2} \tilde{\zeta}_{xxx}+\eps^2 \mathcal{Q}[\tilde{U}]\tilde{v}_x  = f_2 \; ,
\end{array}
\right.
\end{equation*}
with $\widetilde{h}(t,x)=1+\eps\widetilde{\zeta}(t,x)$ and $\widetilde{F}=(f_1,f_2)^{T}\in L^{\infty}\left([0,\frac{T_{max}}{\sqrt{\eps }}],X^s(\R)\right)$. Then for all $t\in[0,\frac{T_{max}}{\sqrt{\eps} }]$, the error ${\bf U}=U-\widetilde{U}=(\zeta,v)^{T}-(\widetilde{\zeta},\widetilde{v})^{T}$ with respect to $U$ given by Theorem \ref{localexistence} satisfies for all $0\leq t\leq T_{max}/\sqrt{\eps}$ the following inequality
$$
\big\vert {\bf U} \big\vert_{L^{\infty}([0,t],X^s(\R))}\displaystyle\leq \sqrt{\eps}\widetilde{C} \Big( \big\vert {\bf U}_{\mid_{t = 0}} \big\vert_{X^s(\R)}+ t\big\vert\widetilde{F}\big\vert_{L^{\infty}([0,t],X^s(\R))}\Big),
$$
where the constant $\widetilde{C}$ is depending on  $\vert U\vert_{L^{\infty}([0,T_{max}/\sqrt{\eps}],X^s(\R))}$ and $\vert \widetilde{U}\vert_{L^{\infty}([0,T_{max}/\sqrt{\eps}],X^{s+1}(\R))}$.
\end{theorem}
\begin{proof} 
The proof consists on the evaluation of $\frac{1}{2}\frac{d}{dt}\big\vert{\bf U}\big\vert^2_{X^s(\R)}$. Knowing that fact, 
by subtracting the equations satisfied by $U=(\zeta,v)^{T}$ and $\widetilde{U}=(\widetilde{\zeta},\widetilde{v})^{T}$, we obtain:
\begin{equation*}
\left\{
\begin{array}{lcl}
\displaystyle \partial_t{\bf U} + A[U]\partial_x{\bf U}= -\big(A[U]-A[\widetilde{U}]\big)\partial_x\widetilde{U} -\widetilde{F},\\
\displaystyle {\bf U}_{\mid_{t = 0}} = U_0-\widetilde{U}_0 \; .
\end{array}
\right.
\end{equation*}
Consequently, a similar energy estimate evaluation as in Proposition \ref{prop1} yields the desired result.
\end{proof}

\subsubsection{Convergence.}
As a conclusion, the following convergence result states that the solutions of the full Euler system, remain close to the ones of the system we are considering, namely system \eqref{boussinesq}, with a better precision as $\eps^3$ is smaller.

\begin{theorem}[Convergence]\label{convergence}
Let $\eps\in(0,1)$, $s>3/2$, and  $U_0=(\zeta_0,\psi_0)^T\in {H^{s+N}}(\R)^2$ satisfying condition \eqref{depthcond} where N is large enough, uniformly with respect to $\eps\in(0,1)$. Moreover, assume $U^{euler}=(\zeta,\psi)^T$ to be a unique solution to the full Euler system \eqref{Zakharovv} that satisfies the assumption of Proposition \ref{consistency}. Then there exists $C$, $T>0$, independent of $\eps$, such that
\begin{itemize}
\item Our new model \eqref{boussinesq} admits a unique solution $U_{xB}=(\zeta_{xB},v_{xB})^T$, defined on $[0, \frac{T}{\sqrt{\eps}}]$ with corresponding initial data $(\zeta^0,v^0)^T$;
\item The error estimate below holds, at any time $0\le t\le T/\sqrt{\eps}$,
$$
\vert (\zeta, v) - (\zeta_{xB},v_{xB}) \vert _{L^{\infty}([0,t];X^s)} \le C \eps^3 t \lesssim \eps^{5/2} \; .
$$
\end{itemize}
\end{theorem}
\begin{proof}
The first point is provided by  the local existence result Theorem \ref{localexistence}. Thanks to Proposition \ref{consistency}, then the solution of the water wave equations $(\zeta,v)^T$ solve our model \eqref{boussinesq} up to a residual $R$ of order $\eps^3$. The error estimation then follows from the stability Theorem \ref{stability}.
\end{proof}
\section{Solitary Waves}\label{solitary-approx}
\subsection{Explicit Solitary Wave Solution of the extended Boussinesq system}
Solitary waves were initially discovered in shallow water by J.S. Russell during his experiments to design a more dynamic canal boat~\cite{OD03}. Many partial differential equations have been derived in the literature to model the solitary wave observed by Russell. Such models are commonly known as the Korteweg-de Vries (KdV) scalar equation for a unidirectional flow or the coupled Boussinesq and Green-Naghdi evolution equations. These famous nonlinear and dispersive models describe the shallow water waves and admit explicit families of solitary wave solutions~\cite{Bous1872,Ray1876,KDV1895,Serre53,Chen98}. 
The explicit solitary solutions of different nonlinear PDE's can be calculated using many methods. One of these methods is replacing the partial differential equation by an ordinary one (ODE) and thus one can look for explicit solutions in terms of particular functions. This replacement can be done by setting a reference traveling wave and hence one look for traveling-wave solutions. In this section, we seek the explicit solution of traveling waves for the extended Boussinesq system. Let us recall that the extended Boussinesq system that we are considering can be written as:
\begin{equation}\label{boussinesq2}
\left\{
\begin{array}{lcl}
\displaystyle\partial_t\zeta+\partial_x(hv)=0\vspace{1mm}\; ,\\
\displaystyle (  1+\eps\mathcal{T}[\zeta]+\eps^2\mathfrak{T} )\partial_t v + \partial_x\zeta+\eps  v \partial_x v   +\eps^2 \mathcal{Q}v =\mathcal{O} (\eps^3) \; ,
\end{array}
\right.
\end{equation}
where $h(t,x)=1+\eps\zeta(t,x)$ and denote by
\begin{equation}\label{exp2}
\mathcal{T}[\zeta]w =-\frac{1}{3h}\partial_x\big((1+3\eps\zeta)\partial_x w\big)=-\frac{1}{3}(1-\eps\zeta)\partial_x\big((1+3\eps\zeta)\partial_x w\big)+\mathcal{O}(\eps^3), \; \mathfrak{T} w = -\frac{1}{45}\partial_x^4w ,  \; \mathcal{Q}v = -\frac{1}{3}\partial_x\big(vv_{xx}-v_x^2\big) \; .
\end{equation}
In order to find solitary wave solutions of the extended Boussinesq system~\eqref{boussinesq2}, we seek solutions in the form of the traveling wave $\zeta(t,x)=\zeta_c(x-ct)$ and $v(t,x)=v_c(x-ct)$ with $\displaystyle{\lim_{|x| \rightarrow \infty} |(\zeta_c,v_c)|(x)=0}$ where the constant $c \in \mathbb{R}$ is the velocity of the solitary wave. Plugging the above Ansatz into eq.~\eqref{boussinesq2} yields:
\begin{equation}\label{boussinesq3}
\left\{
\begin{array}{lcl}
\displaystyle -c\zeta^{'}_c+(h_c v_c)'=0\vspace{1mm}\; ,\\
\displaystyle -c v^{'}_c +\frac{\eps c}{3}\Big((1+3\eps \zeta_c)v^{''}_c\Big)'-\frac{\eps^2 c}{3}\zeta_c v^{'''}_c +\frac{\eps^2 c}{45} v_c^{(5)} +\zeta^{'}_c +\frac{\eps}{2} (v_c^2)' = \frac{\eps^2}{3}\big(v_c v_c^{''} -(v^{'}_c)^2\big)' \ \; .
\end{array}
\right.
\end{equation}
We may now integrate and, using the vanishing condition at infinity to set the integration constant,
we deduce from the first equation:
\begin{equation}\label{soleq1}
-c\zeta_c+ h_c v_c=0\;.
\end{equation}
Using~\eqref{soleq1}, one can deduce that $v_c^{'''}=c \zeta_c^{'''} + \mathcal{O} (\eps)$.  

One can also check the following identity $\zeta_c \zeta_c^{'''}=(\zeta_c \zeta_c^{''})'-\dfrac{1}{2}\big((\zeta_c^{'})^2\big)'$ is true. Using the latter identities into the second equation of~\eqref{boussinesq3}, we may now integrate and, using the vanishing condition at infinity to set the integration constant one can deduce:
\begin{equation}\label{soleq2}
-c v_c +\frac{\eps}{2} v_c^2 +\zeta_c=-\frac{\eps c}{3}v^{''}_c-\eps^2 c\zeta_c v_c^{''}+\frac{\eps^2 c^2}{3}\zeta_c \zeta^{''}_c -\frac{\eps^2 c^2}{6} (\zeta_c^{'})^2-\frac{\eps^2 c}{45} v_c^{(4)}+ \frac{\eps^2}{3}v_c v_c^{''} -\frac{\eps^2}{3}(v^{'}_c)^2.
\end{equation}
One can deduce from~\eqref{soleq1} the following identity: 
\begin{equation}\label{v''exp}
v_c=c \zeta_c -\eps c \zeta_c^2 + \mathcal{O}(\eps^2)\;.
\end{equation}
Using~\eqref{soleq1} into the \emph{l.h.s} of~\eqref{soleq2} and~\eqref{v''exp} into the \emph{r.h.s} of~\eqref{soleq2}, withdrawing all terms of order $\mathcal{O}(\eps^3)$ one can deduce the following equation:
\begin{equation}\label{soleq22}
\zeta_c-\dfrac{c^2\zeta_c}{2(1+\eps\zeta_c)^2}(2+\eps\zeta_c)=-\frac{\eps c^2}{3}\zeta_c^{''} + \frac{\eps^2 c^2}{6}(\zeta_c^{'})^2 +\frac{\eps^2 c^2}{3} \zeta_c \zeta_c^{''}-\frac{\eps^2 c^2}{45} \zeta_c^{(4)}.
\end{equation}
Multiplying~\eqref{soleq22} by $\zeta_c^{'}$ and integrating once again yields,
\begin{equation}\label{soleq23}
\dfrac{\zeta_c^2}{2} \Big( 1-\dfrac{c^2}{1+\eps \zeta_c} \Big)= \frac{\eps c^2}{6} (\eps\zeta_c -1)(\zeta_c^{'})^2-\frac{\eps^2 c^2}{45} \zeta_c^{'''}\zeta_c^{'} + \frac{\eps^2 c^2}{90} (\zeta_c^{''})^2.
\end{equation}
The equation~\eqref{soleq23} is a third order non linear ordinary differential equation. When dropping the $\eps^2$ terms on the \emph{r.h.s} of~\eqref{soleq23}, one gets the analogous ODE for the GN equation which exhibits the analytical solitary wave solution defined in~\eqref{solGN}.
A careful examination reveals that the equation~\eqref{soleq23} does not admit an explicit solution in any appropriate method. In~\cite{Matsuno2015}, the author studied solitary wave solutions of the Hamiltonian formulation of the extended Green-Naghdi equations by performing a singular perturbation analysis. In the latter paper, Matsuno mentioned that his inspection also reveals that the obtained third-order nonlinear differential equation would not have analytical solutions. The aim was to find an exact solitary wave solution of equation~\eqref{soleq23}. However, analytical approaches might not be applied to many nonlinear problems. The explicit solution of the extended Boussniesq~\eqref{boussinesq2} system remain an open problem. An alternative approach is to consider the numerical solution of the equation~\eqref{soleq23}. Therefore, we validate the asymptotic extended Boussinesq model~\eqref{boussinesq2} by comparing its travelling wave solution (computed numerically) with corresponding solution to the full Euler equations, computed using fast and accurate algorithms~\cite{DC14,Tanaka86}. 
\subsection{Numerical Solitary Wave Solution of the extended Boussinesq system}\label{NumSWSec}
In the previous section, the emphasis was on finding an analytic solution for the extended Boussinesq system of equations of the form of a solitary wave. However, many differential equations, especially nonlinear ones of high order, does not admit exact explicit solutions. Instead, numerical solutions must be considered as an alternative way of dealing with these equations.
To this end we compute the solution of~\eqref{soleq23} numerically by employing the Matlab solver \texttt{ode45}. We compare the obtained solutions with the solutions of water-waves equations. 
The latter is computed using the Matlab script of Clamond and Dutykh~\cite{CD2013} where they introduce a fast and precise approach for computing solitary waves solution. 
We compute the solitary waves for our model with three values of velocity, namely $c=1.025, \ c=1.01$ and $c=1.002$. 
In fact, the Matlab script in~\cite{CD2013} offer fast and accurate results but limited to realtively small velocities. We compare the obtained solutions with the ones corresponding to the full Euler system (numerically computed), the original Green-Naghdi system ($\zeta_{GN}$), the Boussinesq system ($\zeta_{B}$) and the KdV equation ($\zeta_{KdV}$). The explicit solution of the original Green-Naghi model has been initially obtained by Serre in~\cite{Serre53} and later on by Su and Gardner~\cite{SuGardner69}:
\begin{equation}\label{solGN}
\eps \zeta_{GN}(x)=(c^2-1)\  \text{sech}^2 \Big(\sqrt{\dfrac{3(c^2-1)}{4c^2 \eps}}\ x\Big)=\eps c^2 \zeta_{KdV}(x)=\eps c^2 \zeta_{B}(x)\;.
\end{equation}
The waves are rescaled so that the Korteweg-de Vries and Boussinesq solutions  do not depend on $c$. Consistently, we set $\eps=1$. By the convergence theorem, the above solutions provide good approximations of the traveling waves of the exact water-waves equations, when $c-1 \approx \eps \ll 1$, that is in the weakly nonlinear regime. 

In fact, in figure~\ref{SWcomp}, one can see clearly as $c-1 \rightarrow 0$ and after re-scaling, the solitary waves tend towards the KdV solution $(\zeta_{KdV})$. Moreover, when zooming in, one can see that the the full Euler system (water-waves) solution is in better agreement with the solution of the extended Boussinesq model rather than the Green-Naghdi one. 
\begin{figure}[H]
	\centering
	\subcaptionbox{Re-sized waves, $c=1.025,\ 1.01,\ 1.002$}
	{\includegraphics[width=0.47\textwidth]{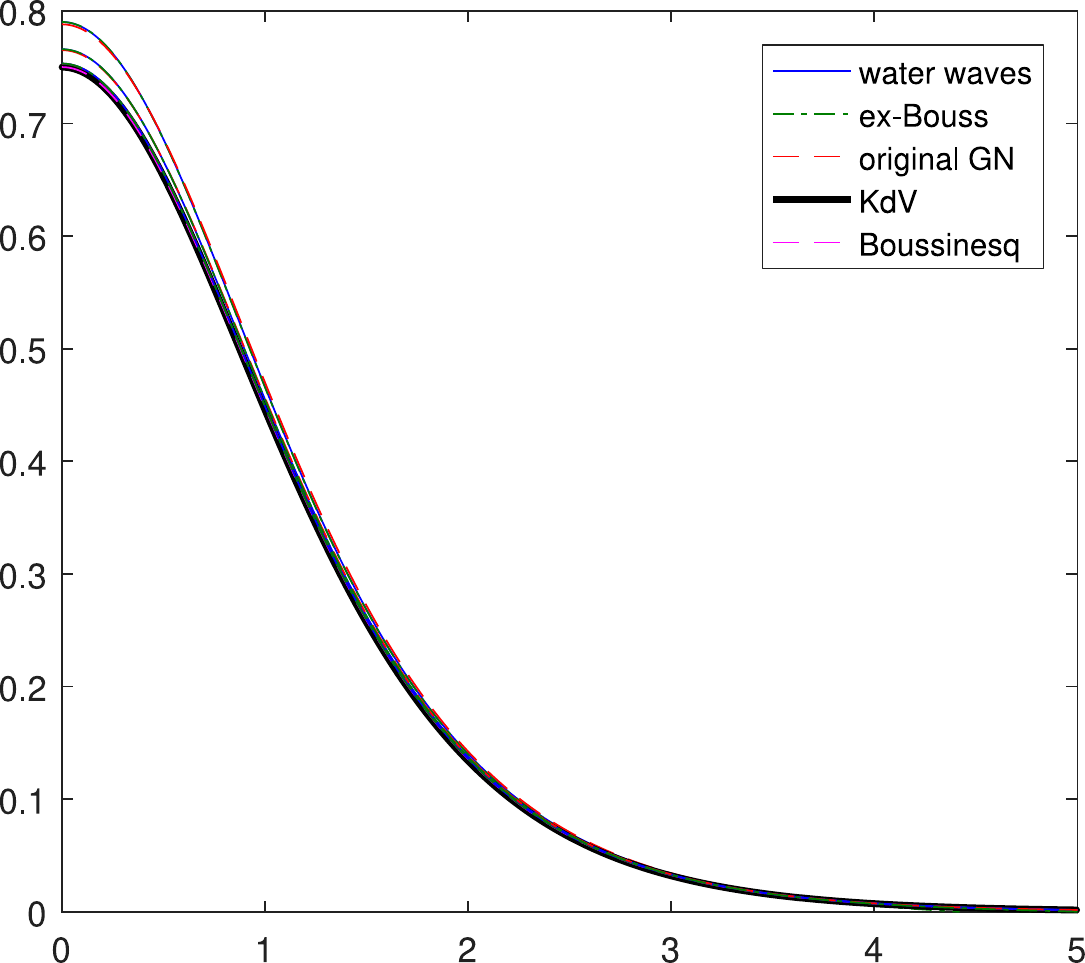}}
	\subcaptionbox{Zoom in}
	{\includegraphics[width=0.48\textwidth]{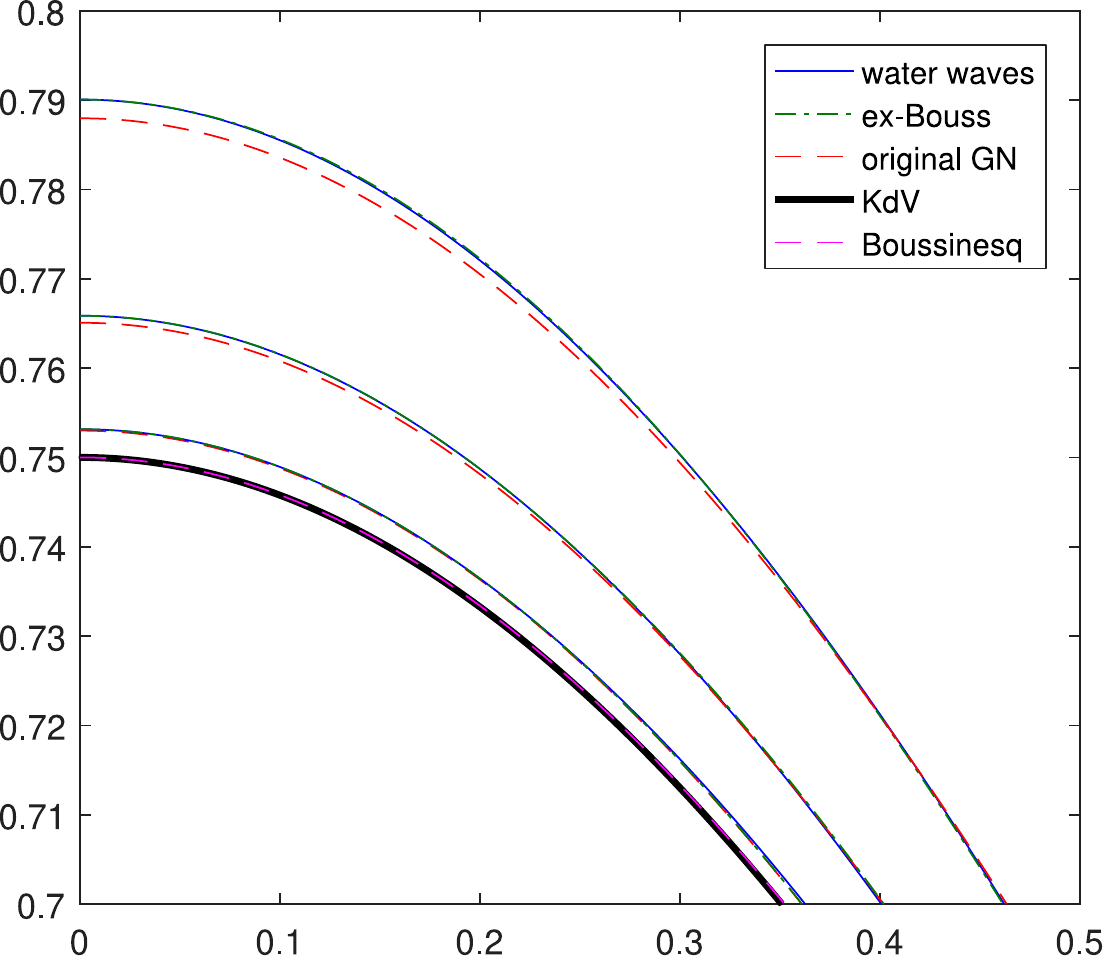}}
	\captionsetup{justification=centering}
	\caption{Comparison of the solitary waves solutions.}
	\label{SWcomp}
\end{figure}
%
%
%
In figure~\ref{convratefig}, we plot in a log-log scale the normalized $l^2$-norm of the difference between the solitary wave solutions of the approximate models and the water-waves solution. The error is computed for different values of $c$.
The extended Bossinesq model exhibit a better convergence rate (quadratic) when compared to the original Green-Naghdi model (linear). This highlight the fact that extended Boussinesq model have a better approximate solution.
\begin{figure}[H]
\includegraphics[scale=0.8]{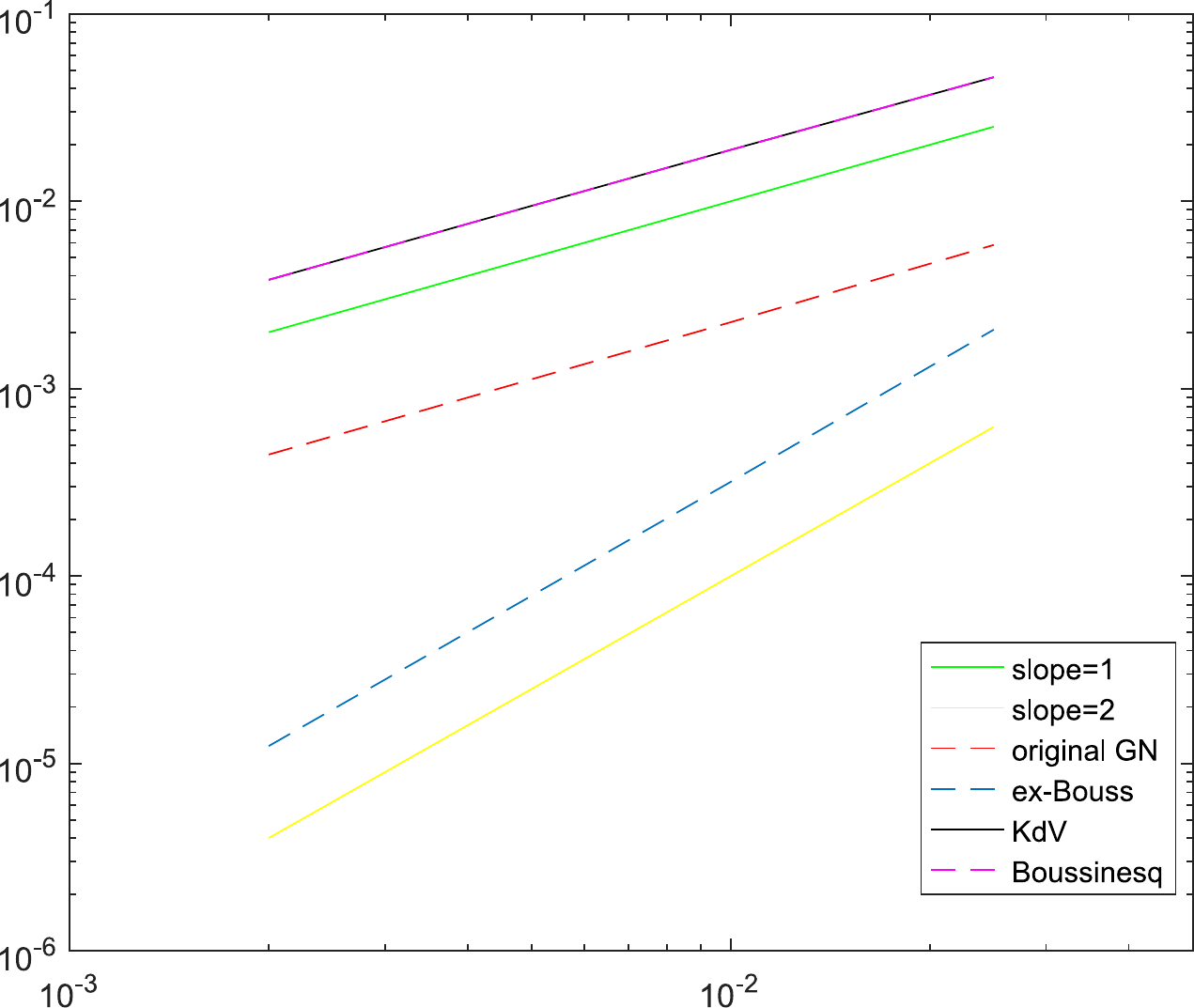}
\caption{Errors as a function of
$c-1$ (log-log plot).}
\label{convratefig}
\end{figure}
\section{Explicit solution with correctors of order $\mathcal{O}(\eps^3)$ for the extended Boussinesq equations}\label{explicit-solitary}
Another approach of dealing with nonlinear PDE's when looking for analytical exact solution is finding instead an explicit solution with correctors. Explicit solutions with correctors for asymptotic water waves models have been obtained in~\cite{IM14,HAI20}. Actually, $H^s$-consistent solutions are obtained to the models in the variable topography case using the analytic solution of the model in the flat topography configuration. In what follows, we find an explicit solution with correctors of order $\mathcal{O}(\eps^3)$ for the extended Boussinesq model~\eqref{boussinesq2} and validate the result numerically.

We start by defining an $H^s$-consistent solution or in other words explicit solution with correctors of order $\mathcal{O}(\eps^3)$.
\begin{definition}
A family $(\zeta,v)$ is $H^s$-consistent on $[0, T/\sqrt{\eps}]$ for the extended Boussinesq equations~\eqref{boussinesq2}, if
\begin{equation}\label{boussHsdef}
\left\{
\begin{array}{lcl}
\displaystyle\partial_t\zeta+\partial_x(hv)=\eps^3 r_1 \vspace{1mm}\; ,\\
\displaystyle (  1+\eps\mathcal{T}[h]+\eps^2\mathfrak{T} )\partial_t v + \partial_x\zeta+\eps  v \partial_x v   +\eps^2 \mathcal{Q}v =\eps^3 r_2 \; ,
\end{array}
\right.
\end{equation}
with $(r_1, r_2)$ bounded in $\Big(L^{\infty} \big([0,\frac{T}{\sqrt{\eps}}], H^s(\mathbb{R})\big)\Big)^2$.
\end{definition}
The standard Boussinesq system can be easily obtained form the extended Boussinesq system~\eqref{boussinesq2} by dropping all terms of order $\mathcal{O}(\eps^2)$. Thus the standard Boussinesq system can be written as:
\begin{equation}\label{standbouss}
\left\{
\begin{array}{lcl}
\displaystyle\partial_t\zeta+\partial_x(hv)=0\vspace{1mm}\; ,\\
\displaystyle \partial_t v -\frac{\eps}{3} \partial_x^2 \partial_t v+ \partial_x\zeta+\eps  v \partial_x v  =\mathcal{O} (\eps^2) \; .
\end{array}
\right.
\end{equation}
\subsection{Explicit solution of the standard Boussinesq system~\eqref{standbouss}}
The standard Boussinesq system enjoys a well known explicit solution of solitary traveling wave $(\zeta_1,v_1)$ of the form:
\begin{equation}\label{stdbousssol}
\left\{
\begin{array}{lcl}
\displaystyle\zeta_{1}(t,x)= \alpha \ \text{sech}^2 \Big(k \ (x-ct)\Big)\vspace{1mm}\; ,\\
\displaystyle v_{1}(t,x)=\dfrac{ c \zeta_1(t,x)}{1+\eps \zeta_1(t,x)}\; ,
\end{array}
\right.
\end{equation}
where $k=\sqrt{\dfrac{3\alpha}{4}}$ and $c=\sqrt{\dfrac{1}{1-\alpha\eps}}$ and $\alpha$ is an arbitrary chosen constant. This explicit solitary wave was already introduced in equation \eqref{solGN} in the previous section~\ref{NumSWSec}. As shown in figure~\ref{SWcomp}, this solution is in good agreement with the water waves solutions in the weakly nonlinear regime.
\begin{theorem}\label{Hsconstheo}
Let $(\zeta_1, v_1)$ be a solution of the standard Boussinesq system~\eqref{standbouss} and $(\zeta_2,v_2)$ solution of the linear equations below:
\begin{equation}\label{linsyst}
\left\{
\begin{array}{lcl}
\displaystyle\partial_t \zeta_2 +\partial_x v_2=0\vspace{1mm}\; ,\\
\displaystyle \partial_t v_2 +\partial_x \zeta_2=f(\zeta_1,v_1) \; ,
\end{array}
\right.
\end{equation}
with 
\begin{equation}\label{gdef}
f(\zeta_1,v_1)=\partial_x \zeta_1 \partial_x \partial_t v_1 +\dfrac{2}{3}\zeta_1 \partial_x^2 \partial_t v_1 + \dfrac{1}{45}\partial_x^4 \partial_t v_1  +\dfrac{1}{3} \partial_x\big(v_1 (v_1)_{xx}-(v_1)_x^2\big),\end{equation}
then $(\zeta,v)=(\zeta_1,v_1)+\eps^2(\zeta_2,v_2)$ is $H^s$-consistent with the extended Boussinesq system~\eqref{boussinesq2}.
\end{theorem}
\begin{proof} First, we would like to mention that we denote by $\mathcal{O}(\eps)$ any family of functions $(f_\eps)_{0<\eps<1}$ such that $(\dfrac{1}{\eps} f_{\eps})_{0<\eps<1}$ remains bounded in $L^{\infty} \big([0,\frac{T}{\sqrt{\eps}}], H^r(\mathbb{R})\big)$, for possibly different values of $r$. We may now proceed in proving the stated result.

If $\zeta$ and $ v$ such that $(\zeta,v)=(\zeta_1,v_1)+\eps^2(\zeta_2,v_2)$  solve the first equation of~\eqref{boussinesq2} up to $\mathcal{O}(\eps^3)$ terms, then
\begin{equation*}
\partial_t \zeta_{1} +\partial_x ((1+\eps \zeta_1) v_1) + \eps^2 \partial_t \zeta_2 +\eps^2 \partial_x v_2 =\mathcal{O}(\eps^3).
\end{equation*}
The first equation of~\eqref{boussinesq2} is satisfied up to $\mathcal{O}(\eps^3)$ terms if and only if:
\begin{equation*} \eps^2 \partial_t \zeta_2 +\eps^2 \partial_x v_2 =\mathcal{O}(\eps^3).
\end{equation*}
Therefore one can take:
\begin{equation*}  \partial_t \zeta_2 + \partial_x v_2 =0.
\end{equation*}
Now, let us recall that the second equation of~\eqref{boussinesq2} can be written as:
\begin{equation*} \partial_t v -\frac{\eps}{3} \partial_x^2\partial_t v-\eps^2\partial_x \zeta \partial_x \partial_t v - \frac{2\eps^2}{3}\zeta \partial_x^2 \partial_t v - \frac{\eps^2}{45}\partial_x^4 \partial_t v+ \partial_x\zeta+\eps  v \partial_x v   -\frac{\eps^2}{3} \partial_x\big(vv_{xx}-v_x^2\big) = \mathcal{O}(\eps^3).
\end{equation*}
We seek $(\zeta_2,v_2)$ such that if $(\zeta,v)=(\zeta_1,v_1)+\eps^2(\zeta_2,v_2)$ and $(\zeta_1,v_1)$ solve the standard Boussinesq equations~\eqref{standbouss}, then the second equation of~\eqref{boussinesq2} is satisfied up to $\mathcal{O}(\eps^3)$ terms if and only if:
\begin{equation*} \eps^2 \partial_t v_2 +\eps^2 \partial_x \zeta_2 =\eps^2 f(\zeta_1,v_1),
\end{equation*}
with $f(\zeta_1,v_1)=\partial_x \zeta_1 \partial_x \partial_t v_1 +\dfrac{2}{3}\zeta_1 \partial_x^2 \partial_t v_1 + \dfrac{1}{45}\partial_x^4 \partial_t v_1  +\dfrac{1}{3} \partial_x\big(v_1 (v_1)_{xx}-(v_1)_x^2\big)$. Therefore, this yields
\begin{equation*}  \partial_t v_2 + \partial_x \zeta_2 = f(\zeta_1,v_1).
\end{equation*}
Hence, the result is directly obtained given the conditions on $\zeta_2$ and $v_2$ in the theorem statement.
\end{proof}
\subsection{Analytic solution for the linear system~\eqref{linsyst}}
In this section, we find the analytic solution for the two transport equations of system~\eqref{linsyst}.
Lets consider first the initial value problem of~\eqref{linsyst}:
\begin{equation}\label{IVPlinsyst}
\left\{
\begin{array}{lcl}
\displaystyle\partial_t \zeta_2 +\partial_x v_2=0\vspace{1mm}\; , \hspace{3cm} \text{if} \ x \in \mathbb{R}, t >0,\\
\displaystyle \partial_t v_2 +\partial_x \zeta_2=f(t,x), \; \hspace{2.25cm} \text{if} \ x \in \mathbb{R}, t >0,\\
\zeta_2(0,x)=\zeta_2^0(x), \ \ v_2(0,x)=v_2^0(x) \hspace{0.45cm} \text{if} \ x \in \mathbb{R},
\end{array}
\right.
\end{equation}
where $\zeta_2^0$ and $v_2^0$ are both given in $C^\infty(\mathbb{R})$. One can equivalently check the following:
\begin{equation}\label{IVPlinsyst2}
\left\{
\begin{array}{lcl}
\displaystyle\partial_t (\zeta_2+v_2) +\partial_x (\zeta_2+v_2)=f(t,x)\vspace{1mm}\; , \hspace{1.95cm} \text{if} \ x \in \mathbb{R}, t >0,\\
\displaystyle \partial_t (\zeta_2-v_2) -\partial_x (\zeta_2-v_2)=-f(t,x), \; \hspace{1.7cm} \text{if} \ x \in \mathbb{R}, t >0,\\
\zeta_2(0,x)=\zeta_2^0(x), \ \ v_2(0,x)=v_2^0(x) \hspace{2.2cm} \text{if} \ x \in \mathbb{R},
\end{array}
\right.
\end{equation}
The analytical solution of both transport equations of system~\eqref{IVPlinsyst2} are:
\begin{equation*}
\zeta_2+v_2=(\zeta_2^0+v_2^0)(x-t)+ \int_0^t f(s,x-t+s) ds,
\end{equation*}
and 
\begin{equation*}
\zeta_2-v_2=(\zeta_2^0-v_2^0)(x+t)- \int_0^t f(s,x+t-s) ds.
\end{equation*}
Thus, one can easily deduce that the analytic solutions of system~\eqref{IVPlinsyst} are given by
\begin{equation}\label{zeta2def}
\zeta_2=\dfrac{1}{2}\Big[(\zeta_2^0+v_2^0)(x-t)+(\zeta_2^0-v_2^0)(x+t)+ \int_0^t f(s,x-t+s) ds-\int_0^t f(s,x+t-s) ds\Big],
\end{equation}
and
\begin{equation}\label{v2def}
v_2=\dfrac{1}{2}\Big[(\zeta_2^0+v_2^0)(x-t)-(\zeta_2^0-v_2^0)(x+t)+ \int_0^t f(s,x-t+s) ds+ \int_0^t f(s,x+t-s) ds \Big].
\end{equation}
\subsection{Explicit solution with correctors for the system of equations~\eqref{boussinesq2}.}
In what follows, we prove that the extended Boussinesq system~\eqref{boussinesq2} enjoys an explicit solution with correctors of order $\mathcal{O}(\eps^3)$.
\begin{theorem}\label{Theo}
Let $(\zeta_1,v_1)$ given by the expressions in~\eqref{stdbousssol} and $f(t,x)$ as defined in~\eqref{gdef}. Lets also consider the initial condition $(
\zeta_0,v_0)=(\zeta_1(0,x),v_1(0,x))+\eps^2(\zeta_2^0,v_2^0)$ where $\zeta_2^0$ and $v_2^0$ are both given in $C^\infty(\mathbb{R})$. Then, the family ($\zeta,v)$ with
\begin{equation}\label{zetadef}
\zeta=\zeta_1 + \dfrac{\eps^2}{2}\Big[(\zeta_2^0+v_2^0)(x-t)+(\zeta_2^0-v_2^0)(x+t)+ \int_0^t f(s,x-t+s) ds-\int_0^t f(s,x+t-s) ds\Big],
\end{equation}
and
\begin{equation}\label{vdef}
v= v_1+\dfrac{\eps^2}{2}\Big[(\zeta_2^0+v_2^0)(x-t)-(\zeta_2^0-v_2^0)(x+t)+ \int_0^t f(s,x-t+s) ds+ \int_0^t f(s,x+t-s) ds \Big],
\end{equation}
is an explicit solution with correctors of order $\mathcal{O}(\eps^3)$ on $[0,\frac{T}{\sqrt{\eps}}]$ for the extended Boussinesq system~\eqref{boussinesq2}.
\end{theorem}
\begin{proof}
Theorem~\ref{Hsconstheo}, gives the $H^s$ consistency result of $(\zeta,v)=(\zeta_1,v_1)+\eps^2(\zeta_2,v_2)$ with the extended Boussinesq system~\eqref{boussinesq2}, where $(\zeta_2,u_2)$ as given in~\eqref{zeta2def} and~\eqref{v2def} is a solution of the linear system~\eqref{linsyst}. Hence the result can be obtained easily.
\end{proof}
\section{Numerical validation}
In this section, we numerically validate the result of Theorem~\ref{Theo}. In fact, we consider the equations given by system~\eqref{boussinesq2} and we compute explicitly the solutions given by~\eqref{zetadef} and~\eqref{vdef}. Then, we compute the residues for both equations after substituting~\eqref{zetadef} and~\eqref{vdef} correspondingly. First we have to set the initial conditions $\zeta_2^0=v_2^0=\exp\Big(-\Big(\dfrac{3\pi x}{10}\Big)^2 \Big)$. We also choose the constant $\alpha=1$. The residues $R_1(\eps)$ and $R_2(\eps)$ of the first and second equation of the system~\eqref{boussinesq2} respectively, are defined as follow:
\begin{equation}\label{Residuesdef}
\left\{
\begin{array}{lcl}
\displaystyle R_1^p(\eps) = \| \partial_t \zeta +\partial_x(hv)\|_{p} \vspace{1mm}\; ,\\
R_2^p(\eps) =\| \displaystyle (  1+\eps\mathcal{T}[h]+\eps^2\mathfrak{T} )\partial_t v + \partial_x\zeta+\eps  v \partial_x v   +\eps^2 \mathcal{Q}v \|_{p} \; .
\end{array}
\right.
\end{equation}
where $p \in \{2,\infty\}$.
The residues $ R_1^p(\eps)$ and $ R_2^p(\eps)$ for $p=1$ and $p=\infty$ are computed for several values of $\eps$, namely $\eps=10^{-1}, \ 10^{-2},\ 10^{-3}, \ 10^{-4}$ and $10^{-5}$, at time $t=1$. The results are summarized in Table~\ref{ResTable} and Figures~\ref{R2curves} and~\ref{Rinfcurves} where we plot in a log-log scale the residues $R_1^{p}$ and $R_2^{p}$ for $p=1$ and $p=\infty$ in terms of $\eps$.
\begin{center}
\begin{table}[H]
\begin{tabular}{ | c | c | c || c | c | c |} 
\hline 
$\eps$ & $R_1^2(\eps)$ & $R_2^2(\eps)$ & $\eps$ & $R_1^{\infty}(\eps)$ & $R_2^{\infty}(\eps)$\\ [0.5ex] 
\hline
1E-1  & 2.70E-02 & 3.80E-03 & 1E-1  & 4.30E-03 & 4.81E-04   \\ 
1E-2  & 2.58E-05 & 2.96E-06 & 1E-2  & 4.17E-06 & 4.10E-07 \\ 
1E-3  & 2.57E-08 & 2.89E-09 & 1E-3  & 4.16E-09 & 4.12E-10 \\ 
1E-4  & 2.57E-11 & 2.88E-12 & 1E-4  & 4.16E-12 & 4.13E-13\\ 
1E-5  & 2.58E-14 & 2.90E-15 & 1E-5  & 4.33E-15 & 5.22E-16 \\ 
\hline
\end{tabular}
\vspace*{5mm}
\caption{The residues $R_1(\eps)$ and $R_2(\eps)$ for $p=2$ (left) and $p=\infty$ (right)}
\label{ResTable}
\end{table}
\end{center}
\begin{figure}[H]
\includegraphics[scale=0.7]{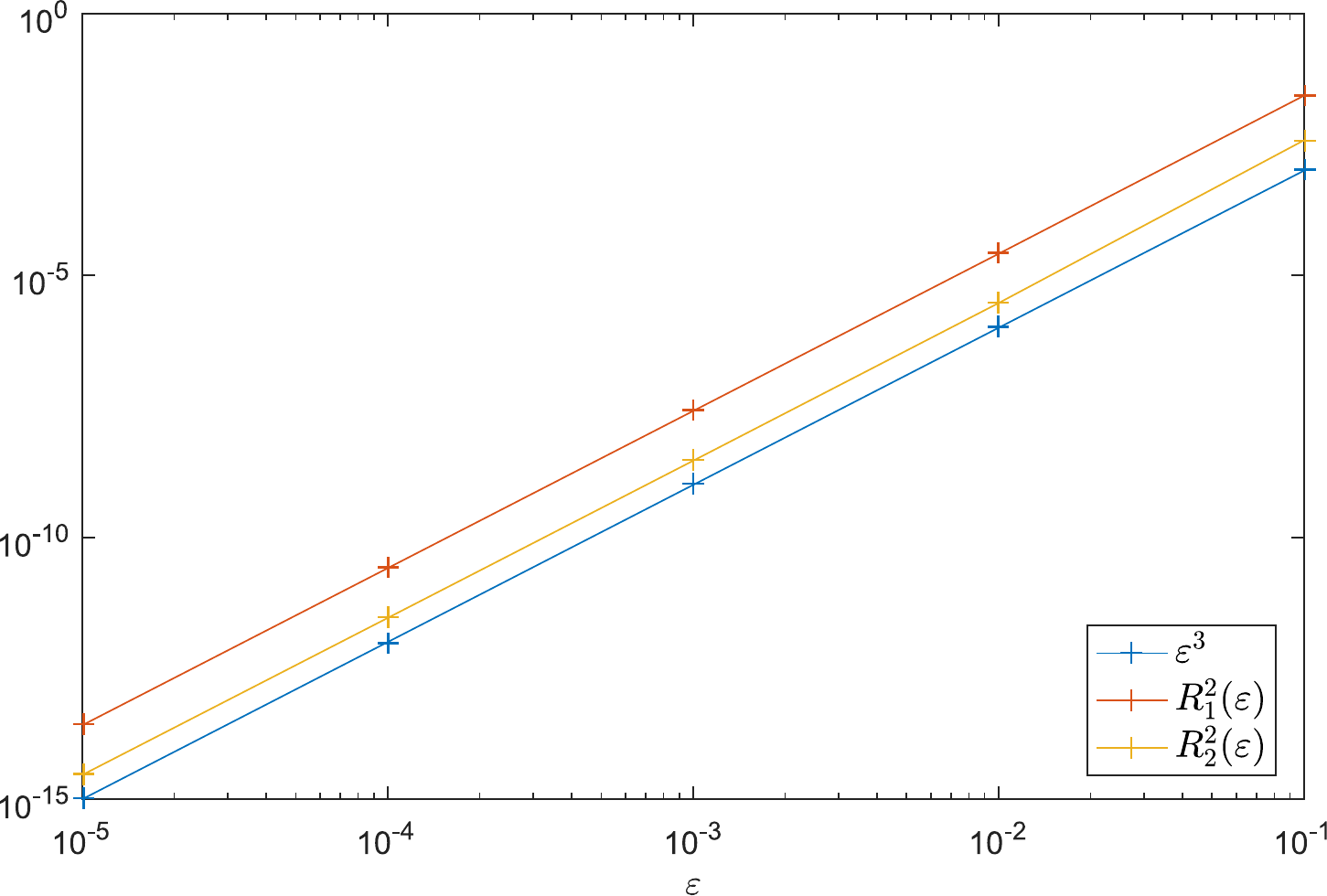}
\caption{The residues $R_1^{\infty}$ and $R_2^{\infty}$ as a function of $\eps$.}
\label{R2curves}
\end{figure}
\begin{figure}[H]
\includegraphics[scale=0.7]{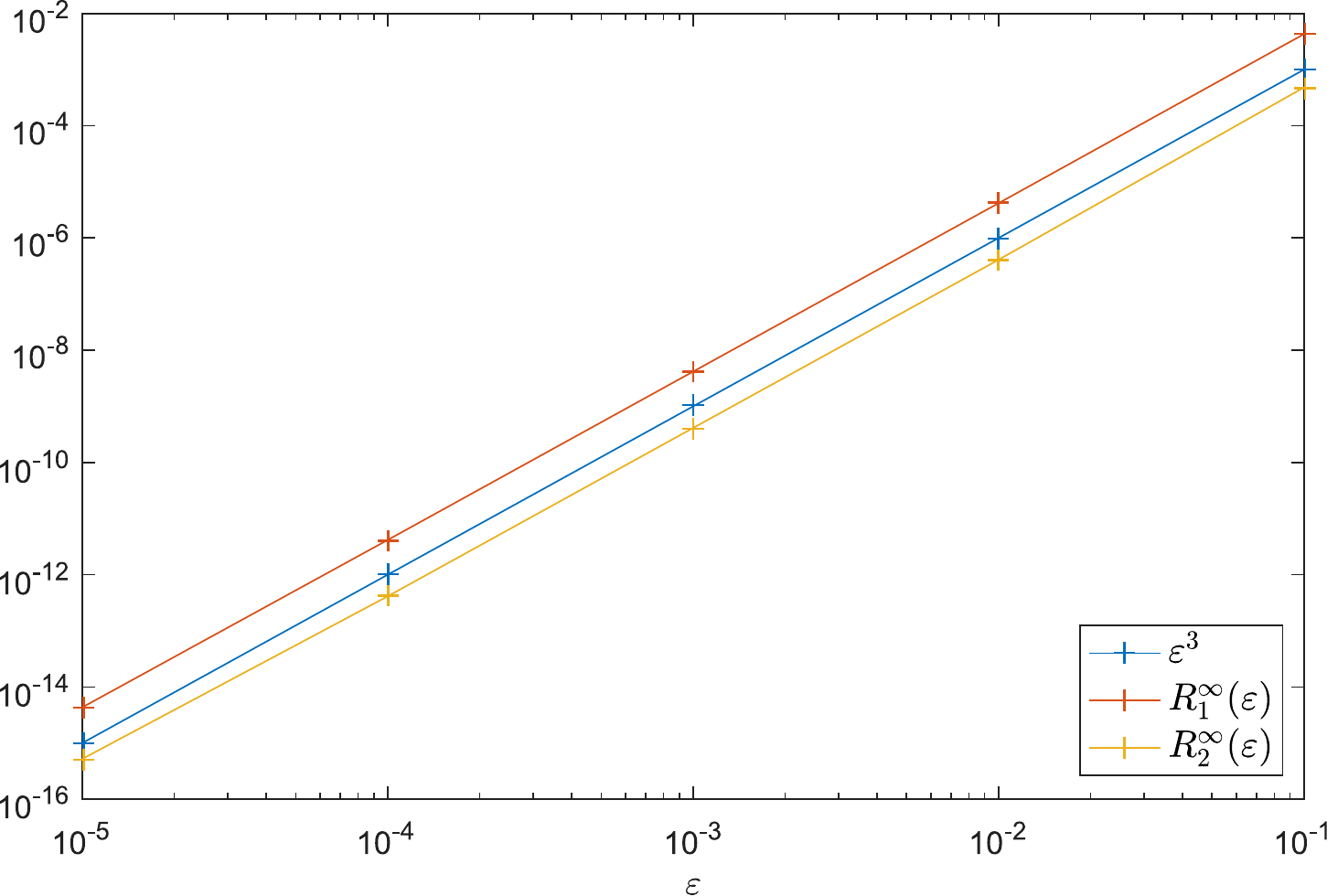}
\caption{The residues $R_1^{\infty}$ and $R_2^{\infty}$ as a function of $\eps$.}
\label{Rinfcurves}
\end{figure}
One clearly sees that the curves of the residues for both $p=1$ and $p=\infty$ are both parallel to $\eps^3$. This shows that the residues convergence rate is $\mathcal{O}(\eps^3)$, which is in total agreement with our theoretical result.
\bibliographystyle{siam}

\end{document}